[1994/12/01]
\documentclass[12pt,reqno]{article}

\title{Long-range last-passage percolation on the line}
\author{Sergey Foss, James B.\ Martin and Philipp Schmidt}
\date{December 29, 2012}
\usepackage{amsmath,amsthm,amsfonts,amssymb}
\usepackage{bbm}
\usepackage{pstricks}
\usepackage{pst-node}
\usepackage{pst-plot}
\usepackage{pst-tree}
\usepackage{pst-grad}
\usepackage{pst-coil}
\usepackage{pst-text}
\usepackage{pst-3d}
\usepackage{pst-eps}
\usepackage[tiling]{pst-fill}
\usepackage{pstricks-add}
\usepackage{graphicx}
\usepackage[format=hang]{caption}

\usepackage[english]{babel}

\newtheorem{theorem}{Theorem}[section]

\newtheorem{lemma}[theorem]{Lemma}
\newtheorem{proposition}[theorem]{Proposition}

\theoremstyle{definition}

\theoremstyle{remark}
\newtheorem{remark}{Remark}[section]
\newtheorem{example}{Example}[section]

\newcommand{\EE}{{\mathbb{E}}}
\newcommand{\PP}{{\mathbb{P}}}
\newcommand{\ZZ}{{\mathbb{Z}}}

\newcommand{\wtw}{{\widetilde{w}}}

\newcommand{\cR}{{\mathcal{R}}}
\newcommand{\cU}{{\mathcal{U}}}
\newcommand{\cA}{{\mathcal{A}}}
\newcommand{\bare}{{\bar{e}}}
\newcommand{\var}{{\operatorname{Var}}}
\newcommand{\essinf}{{\operatorname{ess\,inf}}}

\newcommand{\Axpp}{{A_x^{\scriptscriptstyle{++}}}} 
\newcommand{\Axmp}{{A_x^{\scriptscriptstyle{-+}}}} 
\newcommand{\Axmm}{{A_x^{\scriptscriptstyle{--}}}}

\begin{document}
\maketitle 
\abstract

We consider directed last-passage percolation on the random graph $G =
(V,E)$ where $V = \mathbb{Z}$ and each edge $(i,j)$, for $i < j \in
\mathbb{Z}$, is present in $E$ independently with some probability $p
\in \left( \left. 0,1 \right] \right.$. To every $(i,j) \in E$ we
attach i.i.d.\ random weights $v_{i,j} > 0$. We are interested in
the behaviour of $w_{0,n}$, which is the maximum weight of all
directed paths from $0$ to $n$, as $n \rightarrow \infty$. We 
see two very different types of behaviour, depending on 
whether $\mathbb{E} \left[v_{i,j}^{2} \right] < \infty$ or 
$\mathbb{E} \left[ v_{i,j}^{2}
\right] = \infty$. In the case where $\mathbb{E} \left[ v_{i,j}^{2}
\right] < \infty$ we show that the process has a certain
regenerative structure, and prove a strong law of large numbers and,
under an extra assumption, a functional central limit theorem.
In the situation where
$\mathbb{E} \left[ v_{i,j}^{2} \right] = \infty$ we obtain scaling
laws and asymptotic distributions expressed in terms of a ``continuous
last-passage percolation'' model on $[0,1]$; these are related to 
corresponding results for two-dimensional last-passage percolation
with heavy-tailed weights obtained in \cite{hamblymartin}.
\newline {\scshape AMS 2000 Mathematics Subject Classification: 60K35,
  05C80} \renewcommand{\sectionmark}[1]{}

\section{Introduction}

We study a model of directed last-passage percolation on the integer
line $\ZZ$. 
Consider the random directed graph $G = (\ZZ,E)$, where
every directed edge $(i,j)$ from vertex $i$ to vertex $j>i$ is
present independently with probability $p \in \left( 0,1 \right]$. 
Random structures of this kind have been
used to study community food webs or
task graphs for parallel processing in computer science. In the first
case a link between $i$ and $j$ means that species $j$ preys upon species
$i$. The computational interpretation would be that task $i$ must be
completed before task $j$ can start. Such models have been studied in
\cite{cohenbriandnewman}, \cite{gelenbe}, \cite{isopinewman} and
\cite{newman}, for example.

We consider a model in which weights 
$v_{i,j}$ are attached to all edges $(i,j) \in E$. We are interested in
the asymptotic behaviour of the random variable $w_{0,n}$ that is
defined as follows. For any increasing path $\pi = \left( (i_0,
  i_1), (i_1,i_2), \ldots, (i_{l-1}, i_l) \right)$ from $i=i_0$ to
$j=i_l$ (for some $l \geq 0$) the \textit{weight} of the path is the
sum of the edge weights $\sum_{k=1}^{l} v_{i_{k-1},i_k}$. We define
the weight $w_{i,j}$ to be the maximal weight of a path from $i$ to
$j$. A maximizing path between two points is called a
\textit{geodesic}. If we let $\Pi_{i,j}$ be the set of all paths from
$i$ to $j$, then
\begin{equation} \notag
w_{i,j} = \max_{\pi \in \Pi_{i,j}} \sum_{e \in \pi} v_{e}.
\end{equation}
If there is no path between $i$ and $j$ (which is possible if $p < 1$)
then $w_{0,n}$ will be $-\infty$. We will analyze the behaviour of
$w_{0,n}$ as $n$ tends to infinity. In the parallel processing model, $v_{i,j}$
represents a delay required between the start of task $i$ and the start of task $j$, and  
$w_{0,n}$ represents the overall constraint on the time between the start of task $0$ and the start of 
task $n$. We will write $v$ for a generic random variable whose distribution 
is that of $v_{i,j}$, and we write $F$ for the distribution function of $v$. 

In this paper we study this one-dimensional random graph model with
general weight distributions. The equivalent model with constant weights 
was studied in \cite{fosskonstantopoulos} and
\cite{denisovfosskonstantopoulos}. 
Many features of the model with constant weights remain 
if the random weights have a sufficiently light tail. 
In the case where the weights have finite variance, we 
prove a strong law of large numbers for the passage time $w_{0,n}$,
and under the stronger assumption of a finite third moment
we give a functional central limit theorem. The strategy 
of the proof is similar to that of \cite{denisovfosskonstantopoulos};
we construct a renewal process with the 
property that no geodesic uses an edge which crosses a renewal point. 
In this way much of the analysis can be carried out by considering
the behaviour of the length and weight of individual renewal intervals.
In \cite{denisovfosskonstantopoulos}, 
the definition of renewal points
could be made rather simply in terms of the connectivity properties of
the graph, but here the disorder induced by the random weights
requires a rather more intricate construction. 
We then define a set of auxiliary random variables to
construct an upper bound for $\Gamma_0$, the first
renewal point to the right of the origin, and use them to 
conclude that $\mathbb{E} \left[ \Gamma_0 \right] < \infty$ if
$\mathbb{E} \left[ v^3 \right] < \infty$.
This provides a bound on the second moment of the length of a typical renewal
interval, and we can deduce that the variance that appears in the central limit theorem is finite.

This regenerative structure means that the problem retains
an essentially one-dimensional nature. Most of the edges used in 
an optimal path are short, and the behaviour is qualitatively the
same as one would see in a model with edges of bounded length.
We find an entirely different
behaviour when $\mathbb{E} \left[ v^2 \right] = \infty$. 
Now the passage time $w_{0,n}$ grows super-linearly in $n$, 
and the dominant contribution to the passage time is given 
by the weights of edges whose length is on the order of $n$. 
The appropriate comparison with a bounded-length model is now with a
\textit{two-dimensional} last-passage percolation problem. 
Under the assumption of a regularly varying tail, we prove
scaling laws and asymptotic distributions in terms of 
a ``continuum long-range last-passage percolation'' model 
on the interval $[0,1]$. The construction is closely related to 
that used by Hambly and Martin
\cite{hamblymartin}, who studied (nearest-neighbour) 
last-passage percolation in two
dimensions with heavy-tailed weights. 
There are interesting relationships between such models
and the theory of random matrices with heavy-tailed entries
(see for example \cite{BBP1, BBP2}).

The difference between the behaviour of the model in the cases
$\EE[v^2]<\infty$ and $\EE[v^2]=\infty$ 
can be seen in the simulations in Figure \ref{simfig} in Section \ref{figsection}.

The paper is organized in the following way: in Section 2 we will
present the main results. We will split the results up into two main
sections: one for the case where the weights have a second moment
(Section 2.1) and one where they do not (Section 2.2). The main
results in Section 2.1 are the strong law of large numbers and a
functional central limit theorem for the random variable $w_{0,n}$
giving the weight of the heaviest path from $0$ to $n$. 
Another main result (Theorem \ref{longestedge}) describes scaling laws 
for the length of the longest
edge and the weight of the heaviest edge used on the maximizing path
from $0$ to $n$, and we can use these results to conclude that in
certain situations where
$\mathbb{E} \left[ v^3 \right] = \infty$ a central limit theorem cannot hold. 

The main results for the case $\mathbb{E} \left[ v^2 \right] = \infty$ 
are then given in Section 2.2, along with simulations illustrating the 
scaling limit and the difference in behaviour from the case of weights 
with finite variance.

The proofs for the model with finite second moments can then be found
in Section 3. There we will also briefly discuss a variant of the model where the
edge probabilities are not constant, but depend on the lengths of the
edges (Section 3.5). This model was studied in
\cite{denisovfosskonstantopoulos} but without random
weights.
The proofs for the model with infinite second moments will be given in
Section 4; the structure is closely related to that of the corresponding
results in \cite{hamblymartin}.

\section{Main results}

In this section we introduce the main results. Like the rest of the paper, the results will be split up into two parts: results about the model with finite second moments and results about the model with infinite second moment.

\subsection{Weights with finite second moment}

Here we present the results for the model where the weights $v_{i,j}$ have a finite second moment, i.e. $\mathbb{E} \left[v^{2}\right] < \infty$. The aim is to prove a strong law of large numbers and a functional central limit theorem for $w_{0,n}$.
\begin{theorem} \label{SLLN}
If $\mathbb{E} \left[ v^2 \right] < \infty$ then there exists a constant 
$C\in(0,\infty)$ such that
$$ \frac{w_{0,n}}{n} \xrightarrow[n \rightarrow \infty]{} C \text{ a.s.} $$
and
$$ \frac{w_{0,n}^+}{n} \xrightarrow[n \rightarrow \infty]{} C \text{ in } \mathcal{L}^1. $$
\end{theorem}
In order to state the central limit theorem we will need some more
notation. This will also give an idea of how we want to prove the SLLN
and CLT. We define so-called \textit{renewal points} which will give
the model a regenerative structure. In order to do this we need the following
three events, which depend on a constant $c$ to be chosen later. The constant $c$ will have to be sufficiently small but still satisfy $\mathbb{P} \left[ v < c \right] > 0$.
We define
the random variables $\alpha_{i,j}$, $i < j \in \mathbb{Z}$ to be $1$ if
the edge $(i,j)$ is present and $-\infty$ otherwise. For $x\in\ZZ$, define
\begin{equation} \label{Ax++}
A_x^{\scriptscriptstyle{++}} = \bigcap_{l=1}^{\infty} \left\{ w_{x,x+l} \geq cl \right\},
\end{equation}
\begin{equation} \label{Ax-+} A_x^{\scriptscriptstyle{-+}} =
  \bigcap_{j,l=1}^{\infty} \left\{ \alpha_{x-j,x+l} v_{x-j,x+l} <
    c(l+j) \right\}
\end{equation}
and
\begin{equation} \label{Ax--}
A_x^{\scriptscriptstyle{--}} = \bigcap_{l=1}^{\infty} \left\{ w_{x-l,x} \geq cl \right\}.
\end{equation}
We say that $x$ is a renewal point if 
$A_x^{\scriptscriptstyle{++}} \cap A_x^{\scriptscriptstyle{-+}} \cap
A_x^{\scriptscriptstyle{--}}$ holds. Write $A_x$ for this combined event, and write 
$\cR$ for the set of points such that $A_x$ holds.

Let us explain in words the meaning of the three events used in the definition of the 
set $\cR$. The event $\Axpp$ occurs if for every $y>x$, the optimal path 
from $x$ to $y$ has weight at least $c$ times its length. 
The event $\Axmm$ says the equivalent thing about paths from $y$ to $x$ for $y<x$.
Finally, the event $\Axmp$ says that every edge that contains $x$ in its interior
has weight \textit{less} than $c$ times its length. 

We immediately obtain the following property:
\begin{lemma}\label{basicrenewallemma}
If $x\in\cR$ and $i<x<j$, then the optimal path from $i$ to $j$ passes
through the point $x$. In particular,
\begin{equation}
w_{i,j} = w_{i,x} + w_{x,j}. \label{renewal}
\end{equation}
\end{lemma}
For suppose a path includes an edge $(u,y)$ with $x$ in its interior.
Then we can increase the weight of the path by replacing that 
edge by the union of the optimal paths from $u$ to $x$ and from $x$ to $y$.

A priori the set $\mathcal{R}$ might be empty, finite or
infinite, but the following result will be proved in Sections 3.1
(for $p=1$) and 3.2 (for $p<1$):
\begin{lemma} \label{Rinf} Suppose that $\mathbb{E} \left[ v^2
  \right] < \infty$. There exists $c>0$ such that 
the set $\mathcal{R}$ is infinite with probability 1.
\end{lemma}
The result is stated explicitly with a sufficient condition on $c$ in 
Lemma \ref{combinedlemma}.

We can then denote the points in $\mathcal{R}$ by $\left( \Gamma_n
\right)_{n \in \mathbb{Z}}$ where $\Gamma_0$ is the smallest
non-negative element of $\mathcal{R}$.
\begin{remark}
  If $\mathbb{E} \left[ v^2 \right] = \infty$ then the set
  $\mathcal{R}$ is almost surely the empty set. This will follow from
  the proof of Lemma \ref{Rinf}, see Remark \ref{empty}.
\end{remark}
The final consequence of the definition is
that, as suggested by the name, the set of points $\cR$ forms a 
renewal process. Furthermore, conditional on the points of $\cR$, 
the weights of edges contained within different renewal intervals 
are independent. 
These properties are proved in Lemma \ref{cycles},
and will be central to the structure of the argument that follows.

We now have the necessary notation to state the functional central limit theorem for $w_{0,n}$.
To ensure that the variance is finite, we need the stronger
condition $\mathbb{E} \left[ v^3 \right] < \infty$. 
(see Proposition \ref{var} below).

\begin{theorem} \label{CLT} 
Suppose $\EE\left[v^3\right]<\infty$.
Then there exists $c>0$ such that the following holds. 
Let $\sigma^{2} = \operatorname{Var}
  \left( w_{\Gamma_0,\Gamma_1} - C \left( \Gamma_1 - \Gamma_0 \right)
  \right)$ and $\lambda = \mathbb{P} \left[A_0\right] =\PP(0\in\cR)$.  
Then $\sigma^2 < \infty$ and
$$ \left( l_n(t) = \frac{w_{0,[nt]} - Cnt}{\lambda^{\frac{1}{2}}\sigma\sqrt{n}} , t \geq 0 \right) $$
converges weakly as $n \rightarrow \infty$ to a standard Brownian motion.
\end{theorem}

\begin{remark}
Note that the $c$ corresponds to the 
$c$ in the definition of $\left( \Gamma_n \right)_{n \in \mathbb{Z}}$ and 
$C$ is the constant from Theorem \ref{SLLN}. A sufficient condition for finitenes of 
$\EE \Gamma_0$ in terms of $c$ is given in (\ref{ccond2}). It is not obvious from 
the definitions, but it follows from Theorem \ref{CLT} that, as long as (\ref{ccond2}) holds, 
the quantity $\lambda^{\frac{1}{2}}\sigma$ does not depend on $c$.
\end{remark}

The main idea for the proofs of both the SLLN and the CLT is to use
the regenerative structure induced by the renewal points to represent
$w_{0,n}$ as a random sum of i.i.d.\ random variables in the following
way
\begin{equation} \label{sumw0n}
w_{0,n} = w_{0,\Gamma_0} + \sum_{i=1}^{r(n)} w_{\Gamma_{i-1}, \Gamma_i} + w_{\Gamma_{r(n)},n}
\end{equation}
where $r(n)$ is such that $\Gamma_{r(n)}$ is the largest renewal point
to the left of $n$. We will show in Proposition \ref{cycles} that the
random variables $w_{\Gamma_{i-1},\Gamma_i}$ form an i.i.d.\ sequence,
for $i \geq 1$.

Let $\ell_n$ be the length of
the longest edge and $h_n$ the weight of the heaviest edge used on
the geodesic from $0$ to $n$.
The final result of this section concerns the asymptotic behaviour
of 
$\ell_n$ and $h_n$, under the assumption that the tail of 
the distribution is regularly varying with index $s>2$.  
When $2<s<3$, we can deduce
that the fluctuations of the passage time are of order larger than 
$\sqrt{n}$, and so the central limit theorem cannot be extended to this case.

\begin{theorem} \label{longestedge}
Suppose that the tail of $v$ is regularly varying with index $s>2$, 
in the sense that
\begin{equation}\label{F}
\frac{1-F(tx)}{1-F(x)}\to t^{-s} \text{ as } x\to\infty, \text{ for every } t>0.
\end{equation}
Then we have
\begin{equation} \label{upperlower}
\frac{\log \ell_n}{\log n} \rightarrow \frac{1}{s-1} \text{ in probability as } n \rightarrow \infty
\end{equation}
and the same holds with $\ell_n$ replaced by $h_n$.

Furthermore, the fluctuations of $w_{0,n}$ are of larger order than $n^{\beta}$
for any $\beta<1/(s-1)$, in the sense that for any sequence $y_n$,
\[
\PP(w_{0,n}\in [y_n, y_n+n^\beta]) \to 0 \text{ as } n\to\infty.
\]
In particular, if $2<s<3$ then 
\[
\frac{\operatorname{Var} \left( w_{0,n} \right)}{n} \rightarrow \infty,
\]
and a central limit theorem such as that in Theorem \ref{CLT} cannot hold,
even for individual values of $t$.
\end{theorem}
We prove this theorem in Section 3.4 and also give some examples that show how the behaviour of $\ell_n$ depends on the 
tail of the distribution.

\subsection{Weights with infinite second moment}
\label{infsec}
Here we look at weight distributions that do not have a finite second
moment, i.e. $\mathbb{E} \left[ v^2 \right] = \infty$. 
Under this condition, $w_{0,n}$ grows faster than linearly. 
This can be seen by considering the contribution of the single heaviest
edge in $[0,n]$, and noting that  
the expectation of the maximum of $n^2$ i.i.d.\ random variables with infinite variance
has expectation that grows faster than $n$.
Since $w_{0,n}$ is at least as large as the weight of this single edge, 
we have that $\EE w_{0,n}/n\to\infty$ as $n\to\infty$, and from Kingman's subadditive ergodic 
theorem we can conclude that in fact $\frac{w_{0,n}}{n}\to\infty$ a.s.

We will describe the asymptotic behaviour of $w_{0,n}$, under the assumption 
that the tail of the weight distribution is regularly 
varying with index $s\in(0,2)$, in the sense of (\ref{F}). 
We introduce two useful ways to construct our model in discrete
space and explain how the second construction can be used to define
a corresponding model in continuous space on $[0,1]$. We show that the
passage time $w$ for the continuous model is finite and show
convergence of an appropriately rescaled version of $w_{0,n}$ to
$w$. 

\subsubsection{Discrete model}

We start with the case $p=1$. Since $w_{0,n}$ depends only on
$v_{i,j}$ for $0 \leq i,j \leq n$ it suffices to consider only the
interval $[0,n]$. We can then rescale and consider the set $\left\{
  0,\frac{1}{n}, \ldots , \frac{n-1}{n} , 1 \right\}$ instead of the
interval $[0,n]$. For each $n \in \mathbb{N}$ and $0 \leq i < j \leq
n$, let $v^{(n)}_{i,j}$ be i.i.d.\ with distribution $F$. The weight
of the edge between $\frac{i}{n}$ and $\frac{j}{n}$ is now given by
$v^{(n)}_{i,j}$. We introduce some new notation: for two edges $x =
(i,j)$, $y = (i',j')$ we write $x \sim y$ and say $x$ and $y$ are
compatible if $j \leq i'$ or $j' \leq i$. The edges $x$ and $y$ being
compatible means that they do not overlap and that they can both be
used on a path from $0$ to $1$. As before we define
\begin{equation} \label{w0n1}
w_{0,n} = \max_{\pi \in \Pi_n} \sum_{e \in \pi} v^{(n)}_e
\end{equation}
where $\Pi_n$ is the set of all paths from $0$ to $1$ in $\left\{ 0,\frac{1}{n}, \ldots , \frac{n-1}{n} , 1 \right\}$.

\bigskip

We can think of the same model in the following alternative way: let $M_1^{(n)}
\geq M_2^{(n)} \geq \ldots \geq M_{\binom{n+1}{2}}^{(n)}$ be the order
statistics in decreasing order of the $v^{(n)}_{i,j}$. Let
$Y_1^{(n)}, Y_2^{(n)}, \ldots, Y_{\binom{n+1}{2}}^{(n)}$ be a random
ordering of those edges, chosen uniformly from all the
$\binom{n+1}{2}!$ possibilities. $Y_i^{(n)}$ is the location of the
$i$-th largest weight $M_i^{(n)}$. Now
\begin{equation} \label{C0n}
\mathcal{C}_{0,n} = \left\{ A \subset \left\{1,\ldots,\binom{n+1}{2} \right\} : 
Y_i^{(n)} \sim Y_j^{(n)} \text{ for all } i,j \in A \right\}
\end{equation}
is the random set of admissible paths. Then we have
\begin{equation} \label{w0n2}
w_{0,n} = \max_{A \in \mathcal{C}_{0,n}} \sum_{i \in A} M_i^{(n)}
\end{equation}
which is equivalent to the previous definition of $w_{0,n}$ in (\ref{w0n1}).

\subsubsection{Continuous model}

Following the second approach above, we can define a corresponding continuous model.
Let $W_1,W_2,\ldots$ be a sequence of i.i.d.\
exponential random variables with mean 1 and define, for
$k=1,2,\ldots$, $M_k = \left( W_1 + \ldots + W_k
\right)^{-\frac{1}{s}}$. Let $U_1,U_2,\ldots$ and $V_1,V_2,\ldots$ be
two sequences of i.i.d.\ uniform random variables on $[0,1]$
(independent of the $W_k$). Put $Y_i = (\min(U_i,V_i),\max(U_i,V_i))$
for $i=1,2,\ldots$. The $i$th largest weight $M_i$ will be attached
to the $i$th edge $Y_i$. Similar to (\ref{C0n}) we define
\begin{equation} \notag
\mathcal{C} = \left\{ A \subset 
\left\{ 1,2, \ldots \right\} : Y_i \sim Y_j \text{ for all } i,j \in A \right\}.
\end{equation}
Then we can define a last-passage time for this continuous model
analogously to (\ref{w0n2}) by
\begin{equation} \label{w}
w = \sup_{A \in \mathcal{C}} \sum_{i \in A} M_i.
\end{equation}
A priori the random variable $w$ could be infinite, but we will see in
Theorem \ref{main1} below that it is almost surely finite.

\subsubsection{Convergence results}\label{figsection}

The intuition behind the approximation of the discrete model by the continuous one
is the following 
pair of convergence results.
First, for any finite $k \in \mathbb{N}$ we have
\begin{equation} \label{Yeq}
\left( Y_1^{(n)} , Y_2^{(n)} , \ldots , Y_k^{(n)} \right) 
\xrightarrow[]{d} \left( Y_1 , Y_2 , \ldots , Y_k \right)
\end{equation}
as $n \rightarrow \infty$, where we use the product topology on $([0,1]^2)^k$.

Following  \cite{hamblymartin}, let $a_n =  F^{(-1)}\left(1 - \frac{1}{n}\right)$,
and further let 
$b_n =a_{\binom{n+1}{2}} = F^{(-1)}\left(1 - \frac{1}{\binom{n+1}{2}}\right)$ 
and put $\widetilde{M}_i^{(n)} = \frac{M_i^{(n)}}{b_n}$. 
(As an example, if the weight distribution $F$ is Pareto($s$),
with $F(x) = 1-x^{-s}$ for $x\geq 1$, then $b_n$ grows like $n^{2/s}$. 
More generally under assumption (\ref{F}), $\lim_{n\to\infty} \frac{\log b_n}{\log n} = 2/s)$).
Then from classical results 
in extreme value theory (see for example Section 9.4 of \cite{david})
we have for any
$k \in \mathbb{N}$ that
\begin{equation} \label{Meq}
\left( \widetilde{M}_1^{(n)}, \widetilde{M}_2^{(n)}, \ldots , \widetilde{M}_k^{(n)} \right) 
\xrightarrow[]{d} \left( M_1 , M_2 , \ldots , M_k \right) 
\text{ as } n \rightarrow \infty.
\end{equation}

In this way both the locations and weights of the heaviest edges in the discrete model 
(after appropriate rescaling) are approximated by their equivalents in the continuous model.
We will show that it is the heaviest edges, which make the dominant contribution
to the passage time, and obtain the following convergence result:
\begin{theorem} \label{main1} 
The random variable $w$ in
  (\ref{w}) is almost surely finite. 
If $p=1$ and (\ref{F}) holds, then 
$\frac{{w}_{0,n}}{b_n}
  \rightarrow w$ in distribution as $n \rightarrow \infty$.
\end{theorem}
These heavy edges have length on the order of $n$. 
This is in strong contrast to the behaviour in the case $\EE[v^2]<\infty$,
where the important contribution to the passage time is given by edges of order 1. 
See Figure \ref{simfig} for an illustration of the two types of behaviour.

\begin{figure}[htpb]
\vspace{-1cm}
\begin{minipage}[t]{0.49\textwidth}
\includegraphics[width=0.99\linewidth]{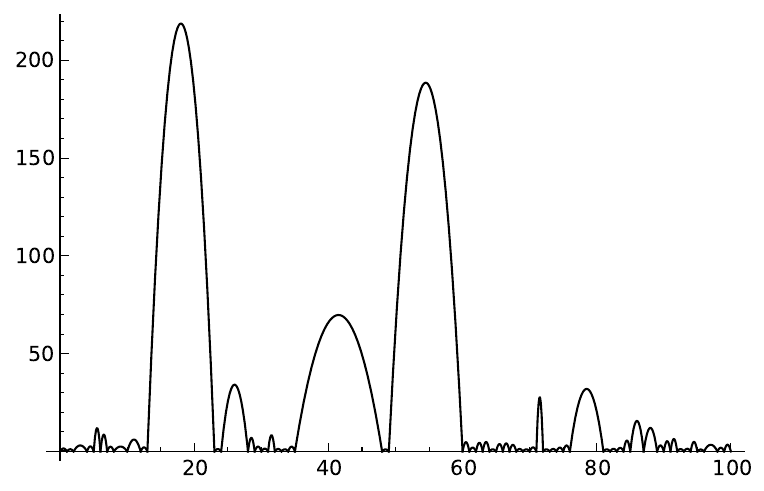}
\end{minipage}
\qquad\qquad
\begin{minipage}[t]{0.49\textwidth}
\includegraphics[width=0.99\linewidth]{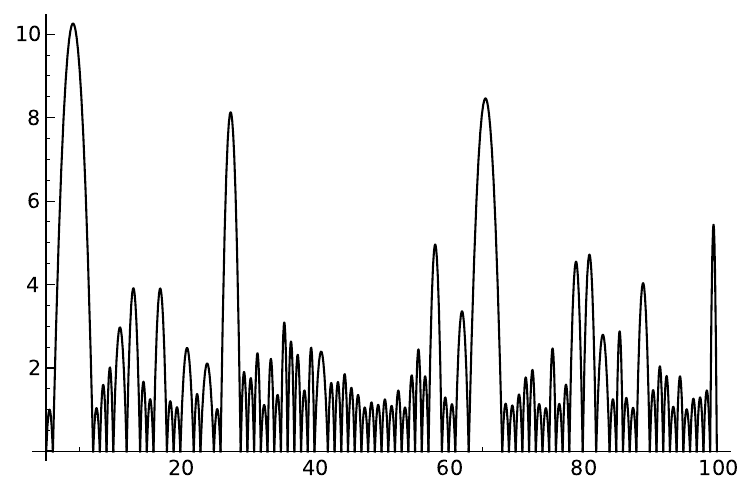}
\end{minipage}
\begin{minipage}[t]{0.49\textwidth}
\includegraphics[width=0.99\linewidth]{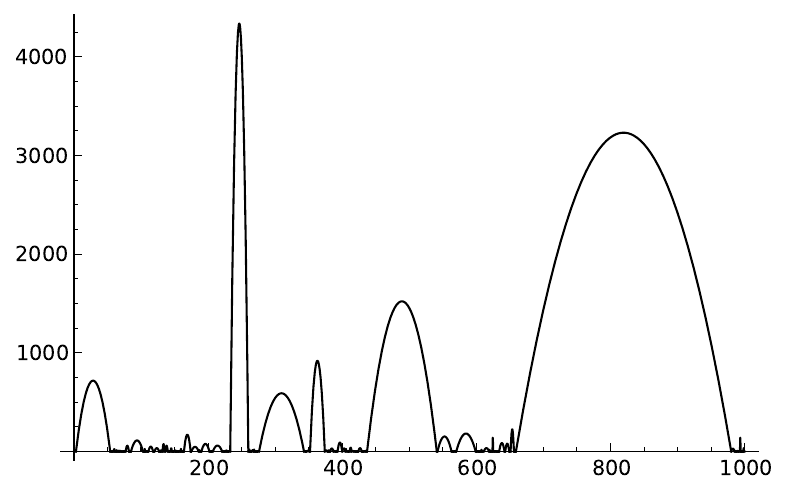}
\end{minipage}
\qquad\qquad
\begin{minipage}[t]{0.49\textwidth}
\includegraphics[width=0.99\linewidth]{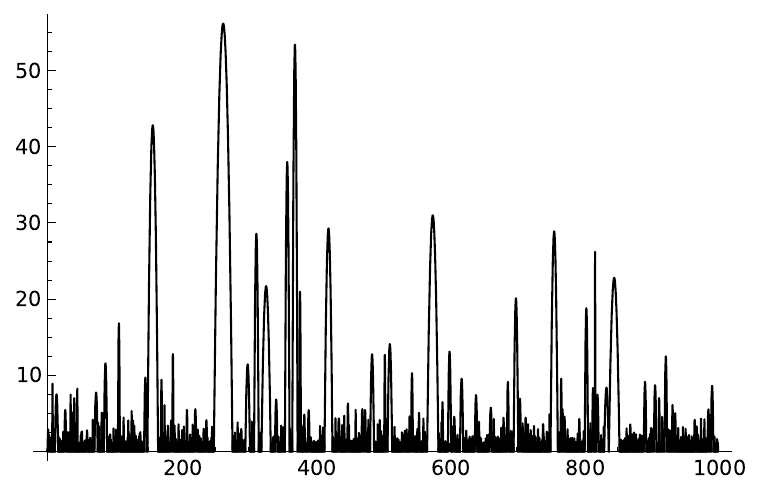}
\end{minipage}
\begin{minipage}[t]{0.49\textwidth}
\includegraphics[width=0.99\linewidth]{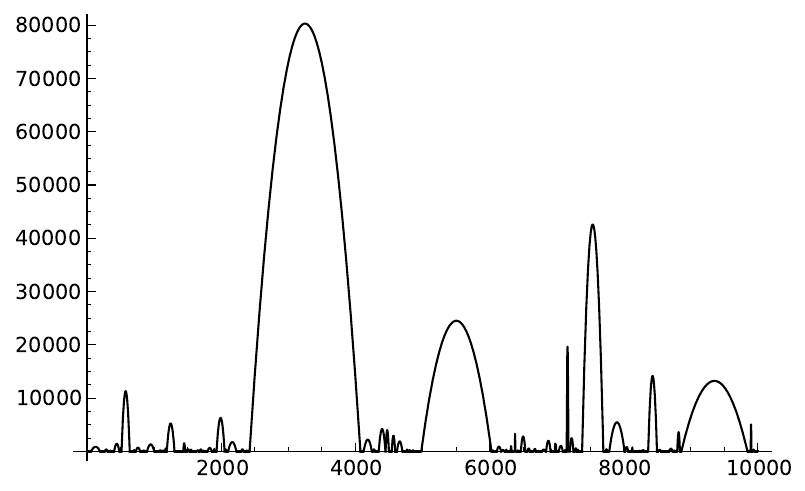}
\end{minipage}
\qquad\qquad
\begin{minipage}[t]{0.49\textwidth}
\includegraphics[width=0.99\linewidth]{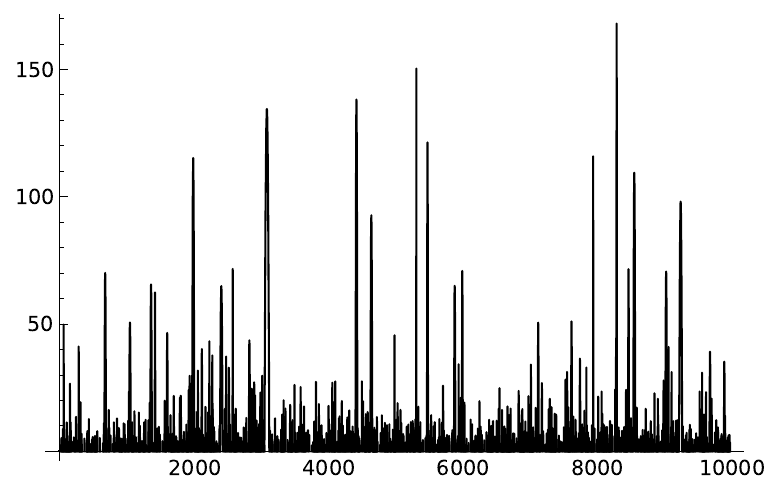}
\end{minipage}
\caption{\label{simfig} 
Simulations for $n=100, 1000$ and $10000$ 
for two weight distributions with $\PP(v>x)=x^{-s}$, $x\geq 1$;
on the left $s=1.5$ and on the right $s=2.5$.
On the left we are in the setting of Section \ref{infsec}.
The scaling limit is clearly visible; visually one 
can hardly distinguish the cases $n=100$ and $n=10000$
(see Remark \ref{main1b} about convergence of the path distribution).
The heaviest edges make an
important contribution to the total weight of the geodesic; their
length is on the order of $n$, and their weight is on the order of
$n^{2/s}$ which is also the order of the total weight of the path.
On the right, the variance of the weight distribution is finite;
the heaviest edges have both length and weight
approximately on the order of $n^{1/(s-1)}$, while the total weight
of the path is on the order of $n$ and obeys a law of large numbers.
(Paths can be generated by a simple dynamic programming algorithm.)}
\end{figure}

This scaling limit extends in a simple way to the case $p<1$,
after taking account of the fact that the total number of edges available
in the interval $[0,n]$ is now on the order of $pn^2/2$ rather than on the order of 
$n^2$:

\begin{theorem} \label{main2}
Let $p\in(0,1]$ and suppose that (\ref{F}) holds. Then
$\frac{p^{-1/s}w_{0,n}}{b_n}\to w$ in distribution as $n\to\infty$.
\end{theorem}

\begin{remark} \label{main1b} 
Although we don't pursue it in detail here, 
one can also prove convergence 
of the optimal path itself in the discrete model to 
that of the continuous model, using an approach 
similar to that in \cite{hamblymartin}. 
For convenience, assume that $F$ is continuous.
Then with probability 1, there exists a unique path $A^{(n)*}\in\mathcal{C}_{0,n}$
which realises the maximal passage time in (\ref{w0n2}). 
One can show that 
in the continuous model there exists a unique set $A^*\in\mathcal{C}$
achieving the supremum in (\ref{w}) (which is therefore in fact a maximum),
and 
$$ \left( \left( Y_i^{(n)} \right)_{i \in \mathbb{N}} , A^{(n)*} \right) 
\xrightarrow[n \rightarrow \infty]{d} 
\left( \left( Y_i \right)_{i \in \mathbb{N}} , A^* \right), $$
where we consider the Euclidean distance on $\mathbb{R}^2$ and the
product topology for the convergence of $\left( Y_i^{(n)} \right)_{i
  \in \mathbb{N}}$, and say that a sequence $(A_k)_{k \in
  \mathbb{N}}$ of subsets of $\mathbb{N}$ converges to a set $A
\subset \mathbb{N}$ if for every $m \in \mathbb{N}$ there exists a $K
\in \mathbb{N}$ such that $A_k \cap \left\{ 1, \ldots, m \right\} = A
\cap \left\{ 1, \ldots, m \right\}$ for all $k \geq K$.

See Theorem 4.2 of \cite{hamblymartin} for an analogous result
in the two-dimensional last-passage case. One can then proceed
to show that in fact the set of endpoints of edges used in the optimal path from 
$0$ to $n$ (rescaled by $n$) converges in distribution under the Hausdorff 
metric to the equivalent object in the continuous model. (Compare
Theorem 4.4 of \cite{hamblymartin}).
\end{remark}

\begin{remark}
Our results do not cover the case where $s=2$ and $\EE v^2=\infty$. 
Since the variance is infinite, it must be the case that
$w_{0,n}$ grows faster than linearly, as noted at the beginning of this section.
On the other hand, by comparison with the scalings obtained for $s<2$, 
the growth must be slower than $n^{1+\epsilon}$ for any $\epsilon>0$. 
It would certainly be interesting to look for appropriate scalings and
limiting distributions in this critical regime. 
\end{remark}

\section{Proofs for the model with $\mathbb{E} \left[ v^2 \right] < \infty$}

In this section we consider the case where the weights $v_{i,j}$ have
a finite second moment, i.e. $\mathbb{E} \left[v^{2}\right] <
\infty$. To avoid degeneracies we assume throughout that $v$ is not a.s.\ constant.
Our main aim is to prove Theorems \ref{SLLN}, \ref{CLT} and
\ref{longestedge}. We start with the model where $p=1$; that is, all
edges are present. First, we show that the set $\mathcal{R}$ of
renewal points is almost surely infinite. Then we
generalize this result to the case where $p \leq 1$. In the following
subsection we will use this result to prove the strong law of large
numbers (Theorem \ref{SLLN}) and the central limit theorem (Theorem
\ref{CLT}) for $w_{0,n}$ for general $p \in \left( \left. 0,1 \right]
\right.$. The next subsection will look at the behaviour of the random
variables $\ell_n$, giving the lengths of the longest edge, and $h_n$,
giving the weight of the heaviest edge, used on the geodesic from $0$
to $n$, see Theorem \ref{longestedge}. We will use these results to
comment on the behaviour of the model if $\mathbb{E} \left[ v^2
\right] < \infty$, but $\mathbb{E} \left[ v^3 \right] =
\infty$. In the last subsection we briefly discuss the case where the
edge probabilities are not constant, but depend on the length of the
edges.

\bigskip                                                                                                                                                                                                                                                                               

\subsection{Proof of Lemma \ref{Rinf} for $p=1$}

When $p=1$, we have $\alpha_{i,j} = 1$ for all $i,j$;
that is, all edges
$(i,j)$, $i,j \in \mathbb{Z}$, are present, and 
in particular there is a path between any two points.

Let $A_x = A_x^{\scriptscriptstyle{++}} \cap
A_x^{\scriptscriptstyle{-+}} \cap A_x^{\scriptscriptstyle{--}}$ be the
event that $x$ is a renewal point. We start with the following Lemma
which is simply Lemma \ref{Rinf} with the additional condition
that $\mathbb{P} \left[ A_0 \right] > 0$. After this Lemma we will
prove in Propositions \ref{independence}, \ref{PAx++} and \ref{PAx-+}
that $\mathbb{P} \left[ A_0 \right] > 0$.
\begin{lemma} \label{Rinf2}
If $\mathbb{P} \left[ A_0 \right] > 0$ then $\mathcal{R}$ is almost surely an infinite set.
\end{lemma}
\begin{proof}
Let $\lambda=\mathbb{P} \left[ A_0 \right]$, 
which is strictly positive by assumption. 
We can approximate the event $A_0$ by an event $A_0^{'}$ that depends only on finitely many of the $v_{i,j}$. In particular, for every $\varepsilon > 0$ there exists $m \in \mathbb{N}$ such that $A_0^{'}$ depends only on $v_{i,j}$ for $-m \leq i < j \leq m$ and $\mathbb{P} \left[ A_0 \Delta A_0^{'} \right] < \varepsilon$. By translation invariance of our model we get that the same is true for any event $A_x$, where $A_x^{'}$ is defined as the translation of $A_0^{'}$ in the natural way. We can for example choose
\begin{align*}
A_x^{'}
&= \left( \bigcap_{l=1}^{m} \left\{ w_{x,x+l} \geq cl \right\} \right) \bigcap \left( \bigcap_{l=1}^{m} \left\{ w_{x-l,x} \geq cl \right\} \right) \\
&\qquad \qquad \qquad \qquad \bigcap \left( \bigcap_{j,l=1}^{m} \left\{ \alpha_{x-j,x+l} v_{x-j,x+l} < c \left(j + l \right) \right\} \right)
\end{align*}
Then 
the events $A_0^{'}$, $A_{2m}^{'}$, $A_{4m}^{'}$, $\ldots$ are i.i.d.\ 
(since they depend on disjoint sets of edges), and we have
\begin{align*}
\mathbb{P} \left[ A_0^{'} \cup A_{2m}^{'} \cup \ldots \cup A_{2(R-1)m}^{'} \right]
&= 1 - \mathbb{P} \left[ \left( A_0^{'} \right)^c \right]^R \\
&\geq 1 - \left( \mathbb{P} \left[ A_0^c \right] + \varepsilon \right)^R \\
&= 1 - \left( 1 - \lambda + \varepsilon \right)^R.
\end{align*}
With this we get
\begin{align*}
\mathbb{P} \left[ A_0 \cup A_{2m} \cup \ldots \right.
&\left. \cup A_{2(R-1)m} \right] = 1 - \mathbb{P} \left[ A_0^c \cap A_{2m}^c \cap \ldots \cap A_{2(R-1)m}^c \right] \\
&\geq 1 - \left( \mathbb{P} \left[ \left( A_0^{'} \right)^c \cap \left( A_{2m}^{'} \right)^c \cap \ldots \cap \left( A_{2(R-1)m}^{'} \right)^c \right] + R\varepsilon \right) \\
&\geq 1 - \left( 1 - \lambda + \varepsilon \right)^R - R\varepsilon.
\end{align*}
For any $\delta > 0$ we can now first choose $R$ large enough such that $\left( 1 - \lambda + \varepsilon \right)^R < \frac{\delta}{2}$ for all small enough $\varepsilon$, and then further choose $\varepsilon > 0$ small enough such that 
also $R\varepsilon < \frac{\delta}{2}$, to get
$$ \mathbb{P} \left[ A_0 \cup A_{2m} \cup \ldots \cup A_{2(R-1)m} \right] \geq 1 - \delta. $$
Since $\delta$ was arbitrary this shows that at least one of the events $A_x$, for $x \geq 0$, holds. In the same way we can show that with probability $1$ for any fixed $y \in \mathbb{Z}$ there exists $x \geq y$ such that $A_x$ holds. This implies that with probability 1 
infinitely many of the $A_x$ hold and therefore $\mathcal{R}$ is almost surely an infinite set. 
\end{proof}

Now it remains to show that the condition of Lemma \ref{Rinf2} is satisfied, i.e. that $\mathbb{P} \left[ A_0 \right] > 0$. To be precise, we have to show that for sufficiently small $c > 0$ with $\mathbb{P} \left[ v < c \right] > 0$ we have $\mathbb{P} \left[ A_0 \right] > 0$. We will do this in four steps: first we show that the events $A_0^{\scriptscriptstyle{++}}$, $A_0^{\scriptscriptstyle{-+}}$ and $A_0^{\scriptscriptstyle{--}}$ are independent and then we will show for each of them that they hold with positive probability for a suitable $c>0$.
\begin{proposition} \label{independence}
For any fixed $x \in \mathbb{Z}$ the events $A_x^{\scriptscriptstyle{++}}$, $A_x^{\scriptscriptstyle{-+}}$ and $A_x^{\scriptscriptstyle{--}}$ are independent.
\end{proposition}
\begin{proof}
  As already mentioned above, the event $A_x^{\scriptscriptstyle{++}}$
  depends only on edges whose left endpoint is at least $x$,
  $A_x^{\scriptscriptstyle{-+}}$ depends only on edges with their left
  endpoint to the left of $x$ and their right endpoint to the right of
  $x$, and $A_x^{\scriptscriptstyle{--}}$ depends only on edges whose right endpoint 
is at most $x$. Since all the weights are i.i.d.\
  this implies the required independence of the events
  $A_x^{\scriptscriptstyle{++}}$, $A_x^{\scriptscriptstyle{-+}}$ and
  $A_x^{\scriptscriptstyle{--}}$.
\end{proof}
\begin{proposition} \label{PAx++} 
If $\mathbb{E} \left[ v \right] <
  \infty$, then for any $c < \mathbb{E} \left[ v \right]$ we have that
  $\mathbb{P} \left[ A_x^{\scriptscriptstyle{++}} \right] > 0$ and
  $\mathbb{P} \left[ A_x^{\scriptscriptstyle{--}} \right] > 0$.
\end{proposition}
\begin{proof}
  Since all the nearest neighbour edges are present we can bound
  $w_{x,x+l}$ for any $l \in \mathbb{N}$ from below by
  $\sum_{j=0}^{l-1} v_{x+j,x+j+1}$ and the $v_{x+j,x+j+1}$ are
  i.i.d. By the strong law of large numbers we have that $\mathbb{P}
  \left[ \bigcap_{l=L}^{\infty} \sum_{j=0}^{l-1} v_{x+j,x+j+1} \geq cl
  \right] \geq \frac{1}{2}$ for large enough $L$. Since only finitely
  many of the $v_{x+j,x+j+1}$ are involved in the events
  $\sum_{j=0}^{l-1} v_{x+j,x+j+1} \geq cl$ for $l < L$ and the
  $v_{x+j,x+j+1}$ are i.i.d.\ there is positive probability that all
  events $\sum_{j=0}^{l-1} v_{x+j,x+j+1} \geq cl$ hold for $l<L$ as
  well. So $\mathbb{P} \left[ A_x^{\scriptscriptstyle{++}} \right] >
  0$.  The proof for $A_x^{\scriptscriptstyle{--}}$ is exactly the
  same.
\end{proof}
\begin{proposition} \label{PAx-+} 
Assume that $v$ is not a
  constant. If $\mathbb{E} \left[ v^2 \right] < \infty$, then for
  every $c$ such that
\[
\essinf[v]<c<\EE[v],
\]
we have $\mathbb{P} \left[
    A_x^{\scriptscriptstyle{-+}} \right] > 0$.
\end{proposition}
\begin{proof}
Note that 
$\essinf[v]<\EE[v]$ since we assume that $v$ is not a.s.\ constant.
We have
\begin{align*}
\mathbb{P} \left[ A_x^{\scriptscriptstyle{-+}} \right]
&= \prod_{j,l = 1}^{\infty} \mathbb{P} \left[ v_{x-j,x+l} < c \left( l+j \right) \right] \\
&= \prod_{i=2}^{\infty} \mathbb{P} \left[ v < c i \right]^{i-1} \\
&= e^{\sum_{i=2}^{\infty} (i-1) \ln \left( 1 - \mathbb{P} \left[ v \geq ci \right] \right)}.
\end{align*}
The exponent is negative, and the RHS is positive if and only the sum converges to
a finite quantity rather than to $-\infty$. Since $\PP(v<ci)>0$ for all $i$
(because $c>\essinf[v]$), and since $\log(1-x)\sim -x$ as $x\to 0$, this holds if and only if 
$\sum i\PP(v\geq ci)$ is finite, which in turn holds if and only if the variance
of $v$ is finite. 
\end{proof}
\begin{remark} \label{empty}
If, on the other hand, $\mathbb{E} \left[ v^2 \right] = \infty$ then $\mathbb{P} \left[ A_x^{\scriptscriptstyle{-+}} \right] = 0$ and therefore $\mathcal{R}$ is empty almost surely.
\end{remark}

\begin{proof}[Proof of Lemma \ref{Rinf}:]
This follows directly from Lemma \ref{Rinf2} and Propositions \ref{independence}, \ref{PAx++} and \ref{PAx-+}.
\end{proof}

\subsection{Proof of Lemma \ref{Rinf} for $p<1$}\label{generalpsection}

Let us now consider the case where $p < 1$. We say that a point $x \in \mathbb{Z}$ is a 
\textit{strongly connected point} if $x$ is connected to every other point by a path. Here we do not consider 
the weights of the edges, so the paths do not have to be optimal. We denote the set of strongly connected points 
by $\mathcal{S}$. In the previous section every point was a strongly connected point since $p=1$. The 
first three from the following four  
results about the strongly connected points have all been shown in Lemmas 5 and 7 in \cite{denisovfosskonstantopoulos},
and the latter one is an exponential analogue of Lemma 6, with a very similar proof. 
\begin{itemize}
 \item the probability that $0$ is a strongly connected point is strictly positive for any $p>0$
 \item there are almost surely infinitely many strongly connected points
 \item the sequence of strongly connected points forms a stationary renewal process
 \item if we let $\ldots, \tau_{-1}, \tau_0, \tau_1, \tau_2, \ldots$ be the sequence of strongly connected points, where $\tau_0$ is the smallest non-negative element of $\mathcal{S}$, then for some $\alpha>0$,
\begin{equation}\label{strongexp}
\EE \left[ e^{\alpha \tau_0} \right] < \infty, \text{ and }
\EE \left[ e^{\alpha (\tau_{i+1}-\tau_i)} \right] < \infty \text{ for all } i.  
\end{equation}
\end{itemize}
By $w_{k,l}$ we denote again the weight of the geodesic from $k$ to $l$. This might now be $-\infty$ if there exists no path between $k$ and $l$ and we are therefore taking the supremum over an empty set. However, if $x$ is a strongly connected point, then $w_{x-j,x+l} > 0$ for all $j,l \in \mathbb{N}$ since we know that there exists a path from any $x-j$ to $x$ and from $x$ to any $x+l$. For $x \in \mathbb{Z}$ let $m(x)$ be the index of the largest strongly connected point such that $\tau_{m(x)} < x$.

\bigskip

The definition of the renewal points is the same as before -- see
(\ref{Ax++}), (\ref{Ax-+}) and (\ref{Ax--}) (now
$\alpha_{i,j} = -\infty$ if the edge $(i,j)$ is not present). By
definition we have that if $x$ is not a strongly connected point then
$w_{x,x+l} = -\infty$ for some $l \geq 1$ or $w_{x-j,x} = -\infty$ for
some $j \geq 1$. So $x$ can only be a renewal point if it is a
strongly connected point. An equivalent of Lemma \ref{Rinf2} still
holds in the case where $p<1$ and we want to prove again that the
condition for Lemma \ref{Rinf2} ($\mathbb{P} \left[ A_0 \right] > 0$)
holds. This will give us Lemma \ref{Rinf}. Again it will be enough to
show that the three events $A_x^{\scriptscriptstyle{++}}$,
$A_x^{\scriptscriptstyle{-+}}$ and $A_x^{\scriptscriptstyle{--}}$ are
independent and that all of them happen with positive probability. The
independence follows directly from the same argument as in Proposition
\ref{independence}. Let $\gamma > 0$ be the density of strongly
connected points and $\delta = \mathbb{E} \left[ w_{\tau_0,\tau_1}
\right]$.
\begin{proposition} \label{PAx++2} 
  If $0<c<\gamma\delta$, then
  $\mathbb{P} \left[ A_x^{\scriptscriptstyle{++}} \right] > 0$,
  $\mathbb{P} \left[ A_x^{\scriptscriptstyle{-+}} \right] > 0$ and
  $\mathbb{P} \left[ A_x^{\scriptscriptstyle{--}} \right] > 0$.
\end{proposition}
\begin{proof}
  First look at the events $A_x^{\scriptscriptstyle{++}}$ and
  $A_x^{\scriptscriptstyle{--}}$. Now not all the nearest neighbour
  edges are present, but we can use the strongly connected points to
  get a similar bound to the one in the proof of Proposition
  \ref{PAx++}. Without loss of generality assume that $x=0$ and note
  that $m(0)=-1$ from the definition. For any $l>\tau_0$ we can write
\begin{equation} \label{sum2}
w_{0,l} \geq w_{0,\tau_0} + \sum_{j=1}^{m(l)} w_{\tau_{j-1},\tau_j} + w_{\tau_{m(l)},l}.
\end{equation}
Fix a $c < \gamma \delta$. Since the strongly connected points form a
stationary renewal process, independent of the weights, and the
density $\gamma$ of strongly connected points is strictly positive,
the terms in the sum are i.i.d.\ and we have both $\frac{m(l)}{l} \rightarrow\gamma$ 
almost surely as $l \rightarrow \infty$, and $\frac1M \sum_{j=1}^M w_{\tau_{j-1}, \tau_j}
\to \delta$ as $M\to\infty$.
So in fact
\[
\frac{1}{l}\sum_{j=1}^{m(l)} w_{\tau_{j-1}, \tau_j}\to\gamma\delta \text{ a.s.\ as } l\to\infty.
\]
Then since $c<\gamma\delta$ by assumption, we have that for some $L$, the event 
\[
w_{0,l}\geq cl \text{ for all }l\geq L
\]
has positive probability. 
But if this event occurs, then we can obtain a realisation for which 
\[
w_{0,l}\geq cl \text{ for all }l\geq 1
\]
occurs by altering the values of only finitely many edges. 
Hence that event also has positive probability, and so $\PP(\Axpp)>0$ as desired.
In exactly the same way, also $\PP(\Axmm)>0$.

Now look at the event $A_x^{\scriptscriptstyle{-+}}$. With the same
arguments as in the previous section we get that for large $L$
\begin{equation} \notag
\mathbb{P} \left[ \bigcap_{j,l=L}^{\infty} \left\{ \alpha_{x-j,x+l} v_{x-j,x+l} \leq c(l+j) \right\} \right] > 0.
\end{equation}
Since there is a probability of $1-p$ for each edge not to be present, i.e. $\alpha_{i,j} = -\infty$, we get
\begin{equation} \notag
\mathbb{P} \left[ \bigcap_{j,l=1}^{L-1} \left\{ \alpha_{x-j,x+l} v_{x-j,x+l} \leq c(l+j) \right\} \right] \geq (1-p)^{(L-1)^2}.
\end{equation}
Hence $\mathbb{P} \left[ A_x^{\scriptscriptstyle{-+}} \right] = \mathbb{P} \left[ \bigcap_{j,l=1}^{\infty} \left\{ \alpha_{x-j,x+l} v_{x-j,x+l} \leq c(l+j) \right\} \right] > 0$ also.
\end{proof}
So we have shown that the condition in Lemma \ref{Rinf2} is still satisfied in the case where $p<1$ and therefore Lemma \ref{Rinf} holds for $p<1$ as well.

To unify the conditions on $c$ for the cases $p=1$ and $p<1$, 
note that if $p=1$ then $\gamma=1$, and that $\EE[v]\leq \delta$. 
Then we can put together the results of the last two sections to give 
the following:
\begin{lemma}\label{combinedlemma}
Let $p\in(0,1]$. If
\begin{equation}\label{ccond1}
\gamma \essinf[v] < c < \gamma \EE[v]
\end{equation}
then $\lambda=\PP(A_0)>0$ and the set $\cR$ is infinite with probability 1.
\end{lemma}

\subsection{Proofs of the SLLN and CLT for general $p \in \left( 0,1 \right]$}

In the previous two sections we have shown that 
under the condition (\ref{ccond1}) on $c$, the set $\mathcal{R}$ of renewal points is
infinite. Now we want to prove a strong law of large numbers and a
central limit theorem for the random variable $w_{0,n}$, see Theorems
\ref{SLLN} and \ref{CLT}. As before we denote the points in
$\mathcal{R}$ by $\ldots, \Gamma_{-1}, \Gamma_0 , \Gamma_1, \ldots$,
where $\Gamma_0$ is the smallest non-negative element of
$\mathcal{R}$. Evaluating the function $w$ at the renewal points
$\Gamma_n$ gives the following equation, related to (\ref{sumw0n}):
\begin{proposition}
For all $m < n$ we have
$$ w_{\Gamma_m , \Gamma_n} = w_{\Gamma_m , \Gamma_{m+1}} + \ldots + w_{\Gamma_{n-1} , \Gamma_n}. $$
\end{proposition}
\begin{proof}
This follows directly from the definition of the renewal points and (\ref{renewal})
\end{proof}
We now want to use the fact stated in this Proposition to prove a strong law of large numbers and a central limit theorem for the random variable $w_{0,n}$. If $w_{\Gamma_m , \Gamma_{m+1}}$, $m \geq 0$ are independent, then for $n\geq\Gamma_0$ we can write
\begin{equation} \label{iid}
w_{0,n} = w_{0,\Gamma_0} + \sum_{i=1}^{r(n)} w_{\Gamma_{i-1} , \Gamma_i} + w_{\Gamma_{r(n)},n}
\end{equation}
where $r(n) = \max \left\{ m : \Gamma_m < n \right\}$ and, since $w_{\Gamma_{i-1},\Gamma_i}$, $i \geq 1$, are i.i.d.\, use then the standard strong law of large numbers and central limit theorem (under moment conditions for the variance of $w_{\Gamma_{i-1} , \Gamma_i}$) applied to the sum in (\ref{iid}) to get corresponding results for $w_{0,n}$. Note that since the density of renewal points $\lambda = \mathbb{P} \left[ A_0 \right]$ is strictly positive we have that $r(n) \sim \lambda n$ for large $n$. So first we want to show that $w_{\Gamma_{i-1} , \Gamma_i}$, $i \geq 1$ are indeed independent.

Define $\mathcal{C}_k = \left( \Gamma_{k} - \Gamma_{k-1}, v_{\Gamma_{k-1} + n,\Gamma_{k-1} + i}, \alpha_{\Gamma_{k-1} + n,\Gamma_{k-1} + i} : 0 \leq n < i \leq \Gamma_{k} - \Gamma_{k-1} \right)$, $k \in \mathbb{Z}$. Then these cycles have a regenerative structure in the following sense:
\begin{lemma} \label{cycles}
The cycles $\left( \mathcal{C}_k , k \in \mathbb{Z} \right)$ are independent and $\left( \mathcal{C}_k , k \in \mathbb{Z} - \left\{ 0 \right\} \right)$ are identically distributed. The process $\left( \Gamma_n \right)_{n \in \mathbb{Z}}$ forms a stationary renewal process. 
\end{lemma}
\begin{proof}
We start with the following observation about the effect the presence of a renewal point at site $x \in \mathbb{Z}$ has on the weights to the left and to the right of $x$. Let $\mathcal{F}_{x}^+$ be the sigma-algebra generated by the $\left( v_{i,j}, \alpha_{i,j} : x \leq i < j \right)$ and let $\mathcal{F}_{x}^-$ be the sigma-algebra generated by the $\left( v_{i,j}, \alpha_{i,j} : i < j \leq x \right)$. These two sigma-algebras are independent as all our weights are independent. But this is still true even if we know that there is a renewal point at $x$. For any $B^- \in \mathcal{F}_x^-$, $B^+ \in \mathcal{F}_x^+$ we have
\begin{align*}
\mathbb{P}
&\left[ \left. B^- \cap B^+ \right| A_x^{\scriptscriptstyle{--}} \cap A_x^{\scriptscriptstyle{-+}} \cap A_x^{\scriptscriptstyle{++}} \right] \\
&\qquad \qquad = \frac{\mathbb{P} \left[ B^- \cap B^+ \cap A_x^{\scriptscriptstyle{--}} \cap A_x^{\scriptscriptstyle{-+}} \cap A_x^{\scriptscriptstyle{++}} \right]}{\mathbb{P} \left[ A_x^{\scriptscriptstyle{--}} \cap A_x^{\scriptscriptstyle{-+}} \cap A_x^{\scriptscriptstyle{++}} \right]} \\
&\qquad \qquad = \frac{\mathbb{P} \left[ B^- \cap A_x^{\scriptscriptstyle{--}} \right] \mathbb{P} \left[ A_x^{\scriptscriptstyle{-+}} \right] \mathbb{P} \left[ B^+ \cap A_x^{\scriptscriptstyle{++}} \right]}{\mathbb{P} \left[A_x^{\scriptscriptstyle{--}} \right] \mathbb{P} \left[A_x^{\scriptscriptstyle{-+}} \right] \mathbb{P} \left[A_x^{\scriptscriptstyle{++}} \right]} \\
&\qquad \qquad = \mathbb{P} \left[ B^- \left| A_x^{\scriptscriptstyle{--}} \right. \right] \mathbb{P} \left[ B^+ \left| A_x^{\scriptscriptstyle{++}} \right. \right] \\
&\qquad \qquad = \mathbb{P} \left[ B^- \left| A_x^{\scriptscriptstyle{--}} \cap A_x^{\scriptscriptstyle{-+}} \cap A_x^{\scriptscriptstyle{++}} \right. \right] \mathbb{P} \left[ B^+ \left| A_x^{\scriptscriptstyle{--}} \cap A_x^{\scriptscriptstyle{-+}} \cap A_x^{\scriptscriptstyle{++}} \right. \right]
\end{align*}
This shows that having a renewal point at $x$ does not introduce any dependence between the weights to the left and the weights to the right of $x$. Now we want to show that if $A_x$ holds we can determine where all the renewal points to the right of $x$ are only by looking at edges with both endpoints to the right of $x$. So assume again that $A_x$ holds. For $y > x$ (and fixed $x$) define the event $\widetilde{A}_y =  A_y^{\scriptscriptstyle{++}} \cap \widetilde{A}_y^{\scriptscriptstyle{-+}} \cap \widetilde{A}_y^{\scriptscriptstyle{--}}$ with
\begin{equation} \notag
\widetilde{A}_y^{\scriptscriptstyle{-+}} = \bigcap_{l \geq 1, 1 \leq j \leq y-x} \left\{ \alpha_{y-j,y+l} v_{y-j,y+l} \leq c(l+j) \right\}
\end{equation}
and
\begin{equation} \notag
\widetilde{A}_y^{\scriptscriptstyle{--}} = \bigcap_{1 \leq j \leq y-x} \left\{ w_{y-j,y} \geq cj \right\}.
\end{equation}
The events $A_y^{\scriptscriptstyle{++}}$, $\widetilde{A}_y^{\scriptscriptstyle{-+}}$ and $\widetilde{A}_y^{\scriptscriptstyle{--}}$ all depend only on edges to the right of $x$. Now we want to show that conditioned on $A_x$ the event $A_y$ holds if and only if the event $\widetilde{A}_y$ holds. On $A_x$ we have
\begin{equation} \label{A_x}
w_{x-j,x} \geq cj \text{ and } \alpha_{x-j,x+l} v_{x-j,x+l} \leq c(l+j) \text{ and } w_{x,x+l} \geq cl \text{ for all } j,l \geq 1.
\end{equation}
Assume that $\widetilde{A}_{x+k}$ holds. Then we have
\begin{align} 
&w_{x+k-j,x+k} \geq cj \text{ and } \alpha_{x+k-j,x+k+l} v_{x+k-j,x+k+l} \leq c(l+j) \notag \\
&\qquad \text{ and } w_{x+k,x+k+l} \geq cl \text{ for all } 1 \leq j \leq k,l \geq 1. \label{widetildeA}
\end{align}
We have to show that we can conclude from this that $A_{x+k}$ holds, i.e.
\begin{align} 
&w_{x+k-j,x+k} \geq cj \text{ and } \alpha_{x+k-j,x+k+l} v_{x+k-j,x+k+l} \leq c(l+j) \notag \\
&\qquad \text{ and } w_{x+k,x+k+l} \geq cl \text{ for all } j,l \geq 1. \label{A}
\end{align}
So take $j > k$. Then we have
\begin{align*}
w_{x+k-j,x+k}
&= w_{x+k-j,x} + w_{x,x+k} \qquad \text{(since } x \text{ is a renewal point)} \\
&\geq c(k-j) + kj \qquad (\text{by } (\ref{A_x}) \text{ and } (\ref{widetildeA})) \\
&= cj
\end{align*}
and also for any $l \geq 1$
\begin{align*}
\alpha_{x+k-j,x+k+l} v_{x+k-j,x+k+l}
&= \alpha_{x-(j-k),x+k+l} v_{x-(j-k),x+k+l} \\
&\leq c(k+l+j-k) \qquad (\text{by } (\ref{A_x})) \\
&= c(l+j)
\end{align*}
So (\ref{A}) holds. This implies that $A_{x+k}$ holds if $\widetilde{A}_{x+k}$ holds. The other implication is obvious.

This shows that for any $m \geq 1$ the cycles $\mathcal{C}_m$, $\mathcal{C}_{m+1}$, $\ldots$ are independent of the position of $\Gamma_{m-1}$ and everything to the left of $\Gamma_{m-1}$. With similar arguments to the ones above we can also show that for any $m \geq 1$ the cycles $\mathcal{C}_{-m}$, $\mathcal{C}_{-m-1}$, $\ldots$ are independent of the position of $\Gamma_{-m}$ and everything to the right of $\Gamma_{-m}$. Overall we get that the cycles $\left( \mathcal{C}_k , k \in \mathbb{Z} \right)$ are independent and, by symmetry, that the cycles $\left( \mathcal{C}_k , k \in \mathbb{Z} - \left\{ 0 \right\} \right)$ are identically distributed.

Then $\Gamma_0, \Gamma_1, \ldots$ and $\Gamma_{-1}, \Gamma_{-2}, \ldots$ are non-stationary (delayed) renewal processes and translation invariance implies that $\left( \Gamma_n \right)_{n \in \mathbb{Z}}$ is a stationary renewal process.
\end{proof}
With this result we can already prove the strong law of large numbers.
\begin{proof}[Proof of Theorem \ref{SLLN}:]
As above, let $r(n)$ be the label of the last renewal point to the left of $n$, 
so that $\Gamma_{r(n)} < n \leq \Gamma_{r(n)+1}$. Then if $n\geq\Gamma_0$,  
\begin{equation}\label{upperlower1}
w_{0,\Gamma_0}+\sum_{i=1}^{r(n)}w_{\Gamma_{i-1}, \Gamma_i}
\leq w_{0,n}
\leq w_{0,\Gamma_0}+\sum_{i=1}^{r(n)+1}w_{\Gamma_{i-1}, \Gamma_i}.
\end{equation}

First we find a linear upper bound for $w_{0,n}$. 
Since the edges in the path from 0 to $n$ cannot overlap, and the sum
of their lengths is $n$, 
we have
\begin{align*}
w_{0,n}
&\leq n+\sum_{0\leq x<y\leq n}[v_{x,y}-(y-x)]_+\\
&\leq n+\sum_{0\leq x<n}Z_x
\end{align*}
where we define $Z_x=\sum_{y>x} [v_{x,y}-(y-x)]_+$.
Note that $Z_x$ are i.i.d.\ and non-negative with 
\begin{align*}
\EE Z_x&=\EE \sum_{y>0}\left[v_{0,y}-y\right]_+\\
&\leq \frac12\EE v^2\\
&<\infty.
\end{align*}
So $\limsup w_{0,n}/n<\infty$ a.s.\, and so from 
the left-hand inequality in (\ref{upperlower1}),
we also have 
\[
\limsup \frac1n \sum_{i=1}^{r(n)}w_{\Gamma_{i-1}, \Gamma_i}<\infty\, a.s.
\]
But $r(n)/n\to\lambda$ a.s.\ as $n\to\infty$,
and the terms $w_{\Gamma_{i-1}, \Gamma_i}$ are i.i.d.\ and 
non-negative for $i\geq 1$. So $\EE w_{\Gamma_{i-1}, \Gamma_i}$ must
be finite. Then finally using again 
the fact that $r(n)/n\to\lambda$ a.s., and the 
law of large numbers on both sides of (\ref{upperlower1}),
we get the a.s. convergence $w_{0,n}/n\to\lambda^{-1}\EE w_{\Gamma_{i-1}, \Gamma_i} $.

To prove the convergence in $\mathcal{L}^1$, we remark that in the particular
case $p=1$ the required convergence (both a.s. and $\mathcal{L}^1$) follows directly from 
Kingman's subadditive ergodic theorem,
since $w_{0,n}$ is superadditive. Then, for $p<1$, we may use the following monotonicity argument.

Note that $w_{0,n} \equiv w_{0,n}(p)$ is an increasing function of $p$ and, in
particular, 
$$
0\le w_{0,n}^{+}(p) \le w_{0,n}^{+}(1).
$$
Since $w_{0,n}^{+1}/n$ converges to a finite constant in $\mathcal{L}^1$, this sequence
is uniformly integrable, and so is the sequence $w_{0,n}^{+}(p)$, for any $p<1$. This
and the a.s. convergence imply convergence in $\mathcal{L}^1$. 
\end{proof}

\begin{remark}
One can show that for non-constant weights there is a strict inequality $C > \widehat{C} \mathbb{E} \left[ v \right]$ where $\widehat{C}$ is the constant corresponding to $C$ in the case where $v \equiv 1$.
\end{remark}

\bigskip

In order to prove the central limit theorem, 
we will need to establish that $\Gamma_1-\Gamma_0$,
the length of a typical renewal interval, has finite variance.
By general results about renewal processes (see for example
Chapter 1, Section 4 in \cite{baccellibremaud}, in particular Remark 4.2.1),
this is equivalent to the property that the ``residual renewal time'' $\Gamma_0$
has finite expectation. In order to obtain that $\EE[\Gamma_0]$ is finite,
an additional condition on the distribution of $v$ is required; 
instead of just a second moment we need that the third moment of $v$ is finite.

\begin{lemma} \label{exp} Suppose $\mathbb{E} \left[ v^3 \right] < \infty$. 
If 
\begin{equation}\label{ccond2}
\gamma\,\essinf[v] < c < 
\gamma\EE\left[ \min_{\tau_0\leq i<j\leq \tau_1} v_{i,j} \right],
\end{equation}
then $\mathbb{E} \left[ \Gamma_{0} \right] < \infty$.
\end{lemma}
\begin{proof}
Recall that the $\tau_r$ are the points of the renewal process of strongly connected points, 
defined at the beginning of Section \ref{generalpsection}, 
with $\dots<\tau_{-1}<0\leq \tau_0<\tau_1<\dots$. So $(\tau_0, \tau_1)$ 
is a typical renewal interval. $\gamma$ is the density of strongly connected points.
Since the process of strongly connected points is independent of the weights $v_{i,j}$,
and the weight distribution is not a.s.\ constant, the RHS of (\ref{ccond2}) is 
strictly greater than the LHS so the set of ``good'' values of $c$ is non-empty. 
Also note that (\ref{ccond2}) implies (\ref{ccond1}), so the conclusion of 
Lemma \ref{combinedlemma} applies.

  We will use an algorithmic construction of $\Gamma_0$ similar to the
  construction in \cite{denisovfosskonstantopoulos} to prove that the
  expectation $\mathbb{E} \left[ \Gamma_{0} \right]$ is finite. Here
  we will not construct $\Gamma_0$ itself, but an upper bound for
  it. We will use the following events
  $A_{x,d}^{\scriptscriptstyle{++}}$,
  $A_{x,d}^{\scriptscriptstyle{-+}}$ and
  $A_{x,d}^{\scriptscriptstyle{--}}$ that are similar to
  $A_x^{\scriptscriptstyle{++}}$, $A_x^{\scriptscriptstyle{-+}}$ and
  $A_x^{\scriptscriptstyle{--}}$ but restricted to certain regions:
\begin{equation} \notag
A_{x,d}^{\scriptscriptstyle{++}} = \bigcap_{l=1}^{d} \left\{ w_{x,x+l} \geq cl \right\},
\end{equation}
\begin{equation} \notag
A_{x,d}^{\scriptscriptstyle{-+}} = \bigcap_{1 \leq l \leq d, j \geq 1} \left\{ \alpha_{x-j,x+l} v_{x-j,x+l} < c(l+j) \right\}
\end{equation}
and
\begin{equation} \notag
A_{x,d}^{\scriptscriptstyle{--}} = \bigcap_{j=1}^{d} \left\{ w_{x-j,x} \geq cj \right\}.
\end{equation}
We now introduce another process $\cU$ related to the renewal process $\cR$. 
Define
\begin{equation} \label{U}
\mathcal{U} = \left\{ x \in \mathbb{Z} : A_{x}^{\scriptscriptstyle{--}} \text{ holds} \right\}.
\end{equation}
A point in $\mathcal{U}$ clearly has to be connected to every point to its left.
In \cite{denisovfosskonstantopoulos} the authors refer to points 
that are connected to every point to their left as \textit{silver points}.
We immediately have $\cR\subseteq\cU$.

We will write $\ldots < \rho_{-2} < \rho_{-1} < 0 \leq \rho_0 < \rho_1 < \ldots$ for
the sequence of points in $\mathcal{U}$, where $\rho_0$ is the smallest
non-negative element of $\mathcal{U}$. 

The following result about $\cU$ is analogous to Lemma \ref{cycles} about $\cR$, 
but is much more straightforward to prove.
For $k\in\ZZ$, define 
\[
\mathcal{D}_k = \left( \rho_{k} - \rho_{k-1}, v_{\rho_{k-1} + n,\rho_{k-1} + i}, 
\alpha_{\rho_{k-1} + n,\rho_{k-1} + i} : 0 \leq n < i \leq \rho_{k} - \rho_{k-1} \right).
\] 
\begin{lemma} \label{cycles2}
The cycles $\left( \mathcal{D}_k , k \in \mathbb{Z} \right)$ are independent and $\left( \mathcal{D}_k , k \in \mathbb{Z} - \left\{ 0 \right\} \right)$ are identically distributed. The process $\cU=\left( \rho_n \right)_{n \in \mathbb{Z}}$ forms a stationary renewal process. 
\end{lemma}
\begin{proof}
Note that if $A_x^{--}$ holds, and $y>x$, then $A_y^{--}$ holds if and only if
$A_{y,y-x}^{--}$ holds. 
Hence given $x\in\cU$, we can find the next $y>x$ such that $y\in\cU$ by finding the smallest $y>x$
such that $A_{y,y-x}^{--}$ holds, and to determine whether the event $A_{y,y-x}^{--}$ holds
we only have to consider edges with both endpoints in the interval $[x,y]$. 
The regenerative structure described in the lemma follows immediately.
\end{proof}

Next we define
\begin{equation} \notag
\mu = \inf \left\{ d > 0 : \mathbbm{1}_{A_{0,d}^{\scriptscriptstyle{-+}} \cap A_{0,d}^{\scriptscriptstyle{++}}} = 0 \right\}.
\end{equation}
The random variable $\mu$ is the smallest distance $d > 0$ such that at least one of $A_{0,d}^{\scriptscriptstyle{-+}}$ and $A_{0,d}^{\scriptscriptstyle{++}}$ fails. Note that $\mu$ may be infinite; this is the case 
precisely if $A_0^{-+}$ and $A_0^{++}$ hold, so that
\begin{equation} \notag
\beta \stackrel{def}{=} \mathbb{P} \left[ \mu = \infty \right] = \mathbb{P} \left[ A_{0}^{\scriptscriptstyle{-+}} \cap A_{0}^{\scriptscriptstyle{++}} \right] > 0.
\end{equation}
The idea of the proof can best be explained using Figure \ref{G0} below.
\begin{figure}[htpb]
\begin{center}
\resizebox{0.85\textwidth}{!}{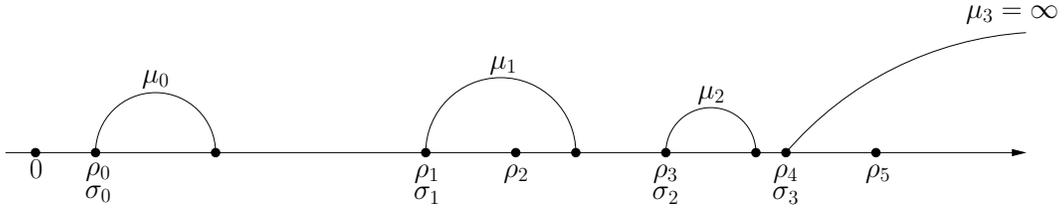}
\caption{\label{G0} Construction of the process $\left( \sigma_n \right)_{n \in \mathbb{Z}}$.
In this case $K=3$; the point $\sigma_3=\rho_4$ is a renewal point, and
provides an upper bound for $\Gamma_0$.}
\end{center}
\end{figure}
We define $\sigma_0=\rho_0$. Now recursively, for each $k\geq0$ we define
\[
\mu_k = \theta_{\sigma_k}\mu = 
\inf \left\{ d > 0 : \mathbbm{1}_{A_{\sigma_k,d}^{\scriptscriptstyle{-+}} \cap A_{\sigma_k,d}^{\scriptscriptstyle{++}}} = 0 \right\},
\]
and 
\[
\sigma_{k+1}=\inf\{x\in\cU: x\geq \sigma_{k}+\mu_k\}.
\]
The set $\{\sigma_0, \sigma_1, \dots\}$ is a subset of $\{\rho_0, \rho_1, \dots\}$. 
We continue until we reach a $K$ such that $\mu_K$ is infinite. Then the corresponding 
$\sigma_K$ must be a point of $\cR$. For certainly $\sigma_K\in\cU$, so the event $A_{\sigma_K}^{--}$ holds.
But also $\mu_K=\infty$, which by definition of $\mu_k$ implies also that
$A_{\sigma_K}^{-+}$ and $A_{\sigma_K}^{++}$ hold. 

In particular $\Gamma_0\leq\sigma_K$, which will serve as the upper bound we require.

Now it also follows from the regenerative properties in Lemma \ref{cycles2} above that
the random variables $\mu_k$ are i.i.d., and their common distribution is the same as that of $\mu$.
So $K=\inf\{k:\mu_k=\infty\}$ is a geometric random variable with parameter $\beta$. 
Also, given $K$, the random variables $\mu_k, 0\leq k<K$ are i.i.d.\
and their common distribution is that of $\mu$ conditioned on $\mu<\infty$
(in particular this does not depend on $K$). 

Since each renewal interval $(\rho_{j-1}, \rho_j)$ has length at least one, 
we also have that $\sigma_k\leq \rho_L$ where $L=\sum_{j=0}^{K-1}\mu_j$. 

We can write $\rho_L$ in the following way
\begin{equation} \label{rhoL}
\rho_L = \rho_0 + \sum_{j=1}^{L} \rho_j - \rho_{j-1}
\end{equation}
with i.i.d.\ $\rho_j - \rho_{j-1}$, $j=1,2,\ldots$. We will use the following Proposition to show that the expectation of $\rho_L$ is finite.
\begin{proposition} \label{ransum}
Let $X_1, X_2, X_3, \ldots$ be an i.i.d.\ sequence of non-negative random variables with finite variance and let $N$ be a non-negative integer valued random variable with finite mean. Then the expectation of $S_N = X_1 + \ldots + X_N$ is finite.
\end{proposition}
\begin{proof}
For $a > \mathbb{E} \left[ X_1
  \right]$, the expectation of
$$ R_a = \sup_{n \in \mathbb{N}} \left( S_n - an \right) $$
is finite whenever $X_1$ has finite variance. This result 
is familiar in the context of queueing theory, saying
that the expected waiting time in a single-server queue
is finite if the service time distribution has finite variance
(see for example Section 2.2 in  \cite{baccellibremaud}).
Therefore,
$$ \mathbb{E} \left[ S_N \right] 
\leq \mathbb{E} \left[ R_a \right] + a \mathbb{E} \left[ N \right] $$
is finite.
\end{proof}
It is therefore enough to show that $\mathbb{E} \left[ L \right] <
\infty$ and $\mathbb{E} \left[ \left( \rho_1 - \rho_0 \right)^2
\right] < \infty$ (note that $\mathbb{E} \left[ \rho_0 \right] <
\infty$ if $\mathbb{E} \left[ \left( \rho_1 - \rho_0 \right)^2 \right]
< \infty$; this is the same renewal process result
we quoted just before Lemma \ref{exp}.). 
The expecation of $L$ is finite if $\mathbb{E} \left[
  \left. \mu \right| \mu < \infty \right] < \infty$ and this will be
proved in Proposition \ref{Emu}. In order to show that $\mathbb{E}
\left[ \left( \rho_1 - \rho_0 \right)^2 \right] < \infty$ we will show
that the following random variable $\nu$, which satisfies $\nu
\stackrel{d}{=} \rho_1 - \rho_0$, has exponential moments for
appropriate $c$:
\begin{equation} \notag
\nu = \inf \left\{ x > 0 : \mathbbm{1}_{A_{x,x}^{\scriptscriptstyle{--}}} = 1 \right\}
\end{equation}
\begin{proposition} \label{Enu}
If $c$ satisfies (\ref{ccond2}) then
$$ \mathbb{E} \left[ e^{\alpha \nu} \right]  < \infty \qquad \text{for some } \alpha > 0 $$
\end{proposition}

\begin{proof}
As above the $\tau_k$ are the strongly connected points with $\tau_{-1}<0\leq \tau_0$,
and $m(x)$ satisfies $\tau_{m(x)}<x\leq\tau_{m(x)+1}$. We immediately have 
\begin{align}
\mathbb{P} \left[ \nu > x \right]
&\leq \mathbb{P} \left[ \nu > \tau_{m(x)} \right] \notag \\
&\leq \mathbb{P} \left[ \left(  A_{\tau_{0}, \tau_{0}}^{\scriptscriptstyle{--}}\right)^c \cap \ldots \cap 
\left( A_{\tau_{m(x)}, \tau_{m(x)}}^{\scriptscriptstyle{--}} \right)^c \right]. \label{firstineq}
\end{align}
Now we claim that if none of the events 
$A_{\tau_{0}, \tau_{0}}^{\scriptscriptstyle{--}},
\dots,
A_{\tau_{m(x)}, \tau_{m(x)}}^{\scriptscriptstyle{--}}$ occur, then for $k=0,1,\dots, m(x)$,
\begin{equation}
\label{sumbound}
\sum_{r=0}^k \min_{\tau_{r-1}\leq i\leq j\leq \tau_r} v_{i,j} < c(\tau_k-\tau_{-1}).
\end{equation}
For suppose (\ref{sumbound}) fails for some value $k\geq 0$ 
(but is true for all smaller values). Then by subtraction, 
\begin{equation}
\label{subtract}
\sum_{r=a}^k \min_{\tau_{r-1}\leq i\leq j\leq \tau_r} v_{i,j} \geq c(\tau_k-\tau_{a-1}) \qquad \forall 0 \leq a \leq k.
\end{equation}
In that case suppose $0\leq x<\tau_k$. For some $a$ with $0\leq a\leq k$ we have
$\tau_{a-1}\leq x<\tau_a$. 

Since the $\tau_r$ are strongly connected points, there exists a path from $x$ to $\tau_k$
which passes through all of $\tau_a,\tau_{a+1}, \dots, \tau_k$, 
and which therefore includes at least one edge within each interval
$[\tau_{a-1}, \tau_a]$, $[\tau_a, \tau_{a+1}],\dots, [\tau_{k-1}, \tau_k]$.

From (\ref{subtract}), this path must have weight at least $c(\tau_k-\tau_{a-1})$,
which is at least $c(\tau_k-x)$.

Since this holds for all $0\leq x\leq \tau_k$, it follows that the event 
$A_{\tau_{k}, \tau_{k}}^{\scriptscriptstyle{--}}$ would have to hold. 

So indeed the event on the RHS of (\ref{firstineq}) implies (\ref{sumbound}),
and so in particular we have
\begin{equation}\label{secondineq}
\PP(\nu>x)\leq \PP\left(
\sum_{r=0}^{m(x)} \min_{\tau_{r-1}\leq i\leq j\leq \tau_r} v_{i,j}
< c(\tau_{m(x)}-\tau_{-1})
\right).
\end{equation}

Since the strongly connected points form a renewal process whose intervals
have exponential moments (see (\ref{strongexp})), we have, for any $\epsilon>0$,
\begin{align}
\label{m-x}
\PP\left(\frac{m(x)}{x}<\gamma-\epsilon\right)&\leq c_1 e^{-c_2x},\\
\label{taum-x}
\PP\left(\frac{\tau_m(x)-\tau_{-1}}{x}>1+\epsilon\right)&\leq c_2 e^{-c_4x},\\
\end{align}
for some constants $c_1$, $c_2$, $c_3$, $c_4$ and all $x\in\ZZ$.

But the quantities $\min_{\tau_{r-1}\leq i<j\leq \tau_r} v_{i,j}$
are non-negative, and i.i.d.\ for $r\geq 1$, and 
we have assumed that $c<\gamma\EE\left[ \min_{\tau_0\leq i<j\leq \tau_1} v_{i,j} \right]$.
Hence for sufficiently small $\epsilon$ and some $c_5$, $c_6$,
\[
\PP\left(
\sum_{r=0}^{\lfloor(\gamma-\epsilon)x\rfloor}
\min_{\tau_{r-1}\leq i<j\leq \tau_r} v_{i,j}
< c(1+\epsilon)x
\right)
\leq c_5 e^{-c_6 x}.
\]
Putting all these together with (\ref{secondineq}), we get that 
$\PP[\nu>x]$ decays exponentially in $x$, as desired.
\end{proof}
Next we want to prove that the expectation of $\mu$, conditioned on $\left\{ \mu < \infty \right\}$ is also finite under suitable moment conditions for $v$.
\begin{proposition} \label{Emu}
If $\mathbb{E} \left[ v^{3} \right] < \infty$ and $c$ satisfies (\ref{ccond2}), then
$$ \mathbb{E} \left[ \left. \mu \right| \mu < \infty \right] < \infty $$
\end{proposition}
\begin{proof}
We have for $d > 0$
\begin{align*}
\mathbb{P} \left[ \mu = d \right]
&= \mathbb{P} \left[ \left( A_{0,d}^{\scriptscriptstyle{-+}} \cap A_{0,d}^{\scriptscriptstyle{++}} \right)^c \cap \left( A_{0,d-1}^{\scriptscriptstyle{-+}} \cap A_{0,d-1}^{\scriptscriptstyle{++}} \right) \right] \\
&\leq \mathbb{P} \left[ \left( A_{0,d}^{\scriptscriptstyle{-+}} \right)^c \cap A_{0,d-1}^{\scriptscriptstyle{-+}} \cap A_{0,d-1}^{\scriptscriptstyle{++}} \right] \\
&\qquad + \mathbb{P} \left[ \left( A_{0,d}^{\scriptscriptstyle{++}} \right)^c \cap A_{\tau_0,d-1}^{\scriptscriptstyle{-+}} \cap A_{\tau_0,d-1}^{\scriptscriptstyle{++}} \right] \\
&\leq \mathbb{P} \left[ \left( A_{0,d}^{\scriptscriptstyle{-+}} \right)^c \cap A_{0,d-1}^{\scriptscriptstyle{-+}} \right] + \mathbb{P} \left[ \left( A_{0,d}^{\scriptscriptstyle{++}} \right)^c \cap A_{0,d-1}^{\scriptscriptstyle{++}} \right] \\
&\leq \mathbb{P} \left[ \sup_{i \geq 1} \left( v_{-i,d} - ci \right) > c d \right] + \mathbb{P} \left[ w_{0,d} < c d \right] \\
&\leq \sum_{i=1}^{\infty} \mathbb{P} \left[ v_{-i,d} > c \left( d + i \right) \right] + \mathbb{P} \left[ w_{0,d} < c d \right]
\end{align*}
By the same arguments as in the proof of Proposition \ref{Enu} we have that the second probability decays exponentially in $d$. Looking at the first probability we get
\begin{align*}
\sum_{d=1}^{\infty} d \sum_{i=1}^{\infty} \mathbb{P} \left[ v_{-i,d} > c \left( d + i \right) \right] 
&= \sum_{l=2}^{\infty} \sum_{j=1}^{l-1} j \mathbb{P} \left[ v > cl \right]\\
&\leq \sum_{l=2}^{\infty} \frac{l^2}{2}  \mathbb{P}[v>cl],
\end{align*}
which is finite if $\mathbb{E} \left[ v^{3} \right]$ is
finite. Therefore, we get that $\mathbb{E} \left[ v^{3} \right] <
\infty$ implies that $\mathbb{E} \left[ \left. \mu \right| \mu <
  \infty \right] < \infty$.
\end{proof}
This completes the proof that $\mathbb{E} \left[ \Gamma_0 \right] < \infty$ 
whenever $c$ satisfies (\ref{ccond2}).\end{proof}

Now we are ready to prove the central limit theorem for $w_{0,n}$ (Theorem \ref{CLT}).
\begin{proof}[Proof of Theorem \ref{CLT}:]
Take any $c$ satisfying (\ref{ccond2}).
  Since under the condition $\mathbb{E} \left[ v^3 \right] <
  \infty$ we have $\mathbb{E} \left[ \Gamma_0 \right] < \infty$ we
  also get $\mathbb{E} \left[ | \Gamma_{-1} | \right] < \infty$. This
  implies that the variance of $\Gamma_1 - \Gamma_0$ is finite (since
  the $\Gamma_n$ form a stationary renewal process, see Remark 4.2.1 in
  \cite{baccellibremaud}). Now we want to show that $\sigma^2 =
  \operatorname{Var} \left( w_{\Gamma_0,\Gamma_1} - C \left( \Gamma_1
      - \Gamma_0 \right) \right)$ is finite. We will prove this in a
  separate Proposition.
\begin{proposition} \label{var}
If $\mathbb{E} \left[ v^3 \right] < \infty$ then 
$\operatorname{Var} \left( w_{\Gamma_0,\Gamma_1} - C \left( \Gamma_1 - \Gamma_0 \right) \right) < \infty$.
\end{proposition}
\begin{proof}
  In order to show that the variance of $w_{\Gamma_0,\Gamma_1} - C
  \left( \Gamma_1 - \Gamma_0 \right)$ is finite, it is enough to show
  that the second moment of this random variable is finite.
\begin{align*}
\mathbb{E}
&\left[ \left( w_{\Gamma_0,\Gamma_1} - C(\Gamma_1 - \Gamma_0) \right)^2 \right] \\
&\qquad = \mathbb{E} \left[ \left( w_{\Gamma_0,\Gamma_1} - C(\Gamma_1 - \Gamma_0) \right)^2 \mathbbm{1}_{\left\{ w_{\Gamma_0,\Gamma_1} \geq C(\Gamma_1 - \Gamma_0) \right\}} \right] \\
&\qquad \qquad \qquad + \mathbb{E} \left[ \left( w_{\Gamma_0,\Gamma_1} - C(\Gamma_1 - \Gamma_0) \right)^2 \mathbbm{1}_{\left\{ w_{\Gamma_0,\Gamma_1} < C(\Gamma_1 - \Gamma_0) \right\}} \right] \\
&\qquad \leq \mathbb{E} \left[ \left( \max_{\Gamma_0 = i_0 < j_0 = i_1 < j_1 = \ldots < j_m = \Gamma_1} \sum_{l=0}^{m} \left[ v_{i_l,j_l} - C(j_l - i_l) \right]_+ \right)^2 \right] \\
&\qquad \qquad \qquad + \mathbb{E} \left[ C^2(\Gamma_1 - \Gamma_0)^2 \right]
\end{align*}
Under the assumption $\mathbb{E} \left[ v^3 \right] < \infty$ we
know that the second expectation is finite. Therefore we will only
consider the first expectation in the following. For the first
expectation we get
\begin{align}
\mathbb{E}
&\left[ \left( \max_{\Gamma_0 = i_0 < j_0 = i_1 < j_1 = \ldots < j_m = \Gamma_1} \sum_{l=0}^{m} \left[ v_{i_l,j_l} - C(j_l - i_l) \right]_+ \right)^2 \right] \\
&\qquad \leq \mathbb{E} \left[ \sum_{\Gamma_0 \leq x < y \leq \Gamma_1} \left[ v_{x,y} - C(y-x) \right]_+^2 \right] \label{1} \displaybreak[0] \\
&+2\EE\sum_{\Gamma_0\leq x<y\leq u<z\leq\Gamma_1}
\left[ v_{x,y} - C(y-x) \right]_+ 
\left[v_{u,z} - C(z-u) \right]_+ 
\label{2}
\end{align}

We will look at the expectations in (\ref{1}) and (\ref{2}) separately. 
For the first one, we can use that the expected length of a typical renewal interval is $\lambda^{-1}$ to give
\begin{align*}
\EE \left[ \sum_{\Gamma_0 \leq x < y \leq \Gamma_1} \left[ v_{x,y} - C(y-x) \right]_+^2 \right]
&\leq \EE \left[ \sum_{\Gamma_0\leq x<\Gamma_1}\sum_{y>x} \left[v_{x,y}-C(y-x)\right]^2_+ \right] \\
&=\lambda^{-1} \sum_{y > 0} \EE \left[ \left[ v_{0,y} - Cy \right]_+^2 \right]\\
&\leq const \cdot \EE\left[ v^3 \right],
\end{align*}
and so the expectation in (\ref{1}) is finite. 

Let us now look at the second expectation, for which we have to sum over pairs of
edges that are in the same renewal interval. For $i\leq j$ 
let $R_{i,j}$ be the event that the set $\{i,i+1,\dots,j\}$ contains
at least one renewal point. Note that $\PP(R_{i,j}^c)=\PP(R_{0,j-i}^c)=\PP(\Gamma_0>j-i)$.

Then define 
\[
s_{r,n}=
\sum_{r\leq x<y\leq u<z\leq n}
\left[v_{x,y}-C(y-x)\right]_+
\left[v_{u,z}-C(z-u)\right]_+
I(R_{x+1,z-1}^c).
\]
Notice that the expression in (\ref{2}) is precisely $\EE s_{\Gamma_0, \Gamma_1}$;
we need to show that this is finite. 
We first aim to show that the expectation of $s_{0,n}$ grows only linearly with $n$. 
To do so we make the following claim, to be proved below: 
for any $x<y\leq u<z$, and any $s,t\geq 0$,
\begin{equation}
\label{renewalclaim}
\PP\left(v_{x,y}\geq t, v_{u,z}\geq s\right)\geq 
\PP\left(v_{x,y}\geq t, v_{u,z}\geq s \big| R_{y,u}^c\right).
\end{equation}
In that case
\begin{align*}
\EE s_{0,n}&=
\EE\sum_{0\leq x<y\leq u<z\leq n}
\left[v_{x,y}-C(y-x)\right]_+\left[v_{u,z}-C(z-u)\right]_+I(R_{x+1,z-1}^c)\\
&\leq
\EE\sum_{0\leq x<y\leq u<z\leq n}
\left[v_{x,y}-C(y-x)\right]_+\left[v_{u,z}-C(z-u)\right]_+I(R_{y,u}^c)\\
&=
\sum_{0\leq x<y\leq u<z\leq n}
\EE\left(\left[v_{x,y}-C(y-x)\right]_+\left[v_{u,z}-C(z-u)\right]_+\big|R_{y,u}^c\right)
\PP(R_{y,u}^c)\\
&\leq
\sum_{0\leq x<y\leq u<z\leq n}
\EE\left(\left[v_{x,y}-C(y-x)\right]_+\left[v_{u,z}-C(z-u)\right]_+\right)
\PP(R_{y,u}^c)\\
&\leq
n\sum_{0<y\leq u<z} 
\EE\left[v_{0,y}-Cy\right]_+\EE\left[v_{u,z}-C(z-u)\right]_+\PP(\Gamma_0>u-y)
\\
&=n\EE \Gamma_0\left(\EE\sum_{y>0}\left[v_{0,y}-Cy\right]_+\right)^2.
\end{align*}
By Lemma \ref{exp}, this gives $\EE s_{0,n}=O(n)$ whenever $\EE(v^3)<\infty$.

Now note that for $n\geq\Gamma_0$ we have
\begin{equation}\label{lower2}
\frac1n \left(s_{0,\Gamma_0}+\sum_{i=1}^{r(n)}s_{\Gamma_{i-1}, \Gamma_i}\right)
\leq \frac1n s_{0,n}
\end{equation}
where as before, we write 
$r(n)$ for the label of the last renewal point to the left of $n$, 
so that $\Gamma_{r(n)} < n \leq \Gamma_{r(n)+1}$. 

By Lemma \ref{cycles}, the quantities $s_{\Gamma_{i-1}, \Gamma_i}$ are i.i.d.\ 
for $i\geq 1$. Hence, if $s_{\Gamma_0, \Gamma_1}$ had infinite mean,
then the left-hand side of (\ref{lower2}) would converge to infinity 
almost surely, since $r(n)/n\to\lambda$ a.s. But then also $s_{0,n}/n$ would converge 
to infinity almost surely, which contradicts the fact that $\EE s_{0,n}/n$ 
is bounded. Hence $s_{\Gamma_0, \Gamma_1}$ has finite mean, which is 
to say that the expectation in (\ref{2}) is finite. 

This completes the proof, subject to the claim (\ref{renewalclaim})
which we now justify.
We consider the dependence of the event 
$R_{y,u}$ on the weights $v_{x,y}$ and $v_{u,z}$. 

Take $r\in\{y,y+1,\dots,u\}$.
From the definition of the set of renewal points $\cR$,
it's easy to see that the event $\{r\in\cR\}$ 
is an increasing event as a function of $v_{x,y}$ and $v_{u,z}$;
that is, if $r\in\cR$ and we increase the values of $v_{x,y}$ 
or $v_{u,z}$ while leaving all other weights the same,
then it remains the case that $r\in\cR$.
But $R_{y,u}=\bigcup_{y\leq r\leq u} \{r\in\cR\}$,
so $R_{y,u}$ is also an increasing event as a function of
$v_{x,y}$ and $v_{u,z}$.

Since these events depend only on the weights $v_{i,j}$ 
and the indicator variables $\alpha_{i,j}$ determining which
edges are present,
and since these quantities are all independent,
it follows that
the distribution of $(v_{x,y}, v_{u,z})$ 
conditioned on $R_{y,u}$ dominates the
unconditioned distribution, which is equivalent to (\ref{renewalclaim}).
\end{proof}
With Proposition \ref{var} established, 
the rest of the argument to prove the central limit theorem in Theorem \ref{CLT}
is analogous to that in \cite{denisovfosskonstantopoulos}, see proof of Theorem 2 (pp. 20-22),
using Donsker's theorem and the continuous mapping theorem 
(and the fact that the fraction of renewal points between $0$ and $[nt]$ 
converges to the deterministic function $\lambda t$).
\end{proof}

\subsection{Length of the longest edge}
In this section we analyze the asymptotic behaviour of
$\ell_n$ and $h_n$, the length of the longest edge 
and the weight of the heaviest edge used 
on the 
geodesic between $0$ and $n$, and prove Theorem \ref{longestedge}. 

We are working under the assumption (\ref{F}) that $F$ is regularly varying with index $s$. 
Define $f(x)$ by $1-F(x)=x^{-s}f(x)$, so that $f$ is a slowly varying function,
i.e.\ $f(tx)/f(x)\to 1$ as $x\to\infty$, for any $t>0$. 

We start with some general results about regularly varying functions that
will be useful throughout this section. Let $g(z) = z^{-s} f(z)$ be a
regularly varying function with index $s>1$ (i.e.\ $f$ is slowly
varying). Then we have
\begin{equation} \label{asymptotic1}
\int_x^{\infty} g(z) dz \sim \frac{x^{-s+1}}{s-1} f(x).
\end{equation}
See for example Proposition 1.5.10 in \cite{binghamgoldieteugels}. From the Representation Theorem (Theorem 1.3.1 in \cite{binghamgoldieteugels}) it follows that we can choose a function $r_0(x)$, depending on $f$, that increases 
to infinity but does so 
slowly enough that
$$ \sup_{x \leq y \leq xr_0(x)} \frac{f(y)}{f(x)} \xrightarrow[]{x \rightarrow \infty} 1; $$
then for this $r_0$
\begin{equation} \label{asymptotic2}
\int_x^{xr_0(x)} g(z) dz \sim \frac{x^{-s+1}}{s-1} f(x)
\end{equation}
From (\ref{asymptotic1}) and (\ref{asymptotic2}) we get that for any function $r(x) \leq \infty$ such that $r(x) \rightarrow \infty$ as $x \rightarrow \infty$ the following holds
\begin{equation} \label{asymptotic3}
\int_x^{xr(x)} g(z) dz \sim \frac{x^{-s+1}}{s-1} f(x)
\end{equation}

\begin{proof}[Proof of Theorem \ref{longestedge}:]
We start by proving the limit in (\ref{upperlower}). We start with an upper bound, 
and aim to show that if $\beta>1/(s-1)$, the optimal path is unlikely to use an edge as long as $n^\beta$.

\noindent \textbf{Upper bound.} 
For any edge $e=(x,y)$, write $|e|=y-x$ for the length of the edge. 
Write $w^-_{i,j}$ for the maximal weight of a path from $i$ to $j$ not using the edge $(i,j)$ itself. 

\begin{lemma}\label{expdecaylemma}
Fix $\beta\in(0,1)$. 
\begin{itemize}
\item[(i)]For some $c_1$ and $M>0$,
\[
\PP(w^-_{0,m}\leq mM)\leq e^{-c_1 m} \text{ for all } m.
\]
\item[(ii)]For some $c_2$ and $M>0$, for all $n$,
\[
\PP\left(w^-_{x,y}\leq M(y-x)\text{ for some }0\leq x<y\leq n \text{ with } y-x\geq n^\beta\right)
\leq e^{-c_2 n^\beta}, 
\]
and
\begin{multline}
\PP\big(\text{The geodesic from 0 to $n$ uses an edge $e$ with }
|e|\geq n^\beta \text{ and } v_e\leq M|e|\big)
\\
\label{longedgebound}
\leq e^{-c_2 n^\beta}.
\end{multline}
\end{itemize}
\end{lemma}

\begin{proof}
Property (i) is immediate for $p=1$, since the quantity $w^-_{0,m}$ is bounded from above 
by the sum of $m$ i.i.d.\ non-negative random variables. For $p<1$ we can do something
analogous using the strongly connected points. From (\ref{strongexp}), 
the distance between successive strongly connected points has an exponentially decaying tail,
and so the probability that there exist fewer than $m\gamma/2$ strongly connected points between
$0$ and $m$ decays exponentially in $m$, where $\gamma$ is the density of strongly connected points.
If there exist at least $m\gamma/2$ such points, then there is a path from 0 to $m$ containing 
at least $m\gamma/2$ edges. 

But the weights are i.i.d., bounded below and with positive mean, so for appropriately chosen $M$ 
the probability that their sum is less than $Mm$ decays exponentially, as required for (i).

For the first part of (ii), simply sum (i) over all appropriate values of $x$ and $y$. 
This introduces an extra factor of $n^2$, but this can be removed by replacing
$c_1$ with sufficiently small $c_2<c_1$ (using the fact that for all $n$ the
probability concerned is strictly less than 1). 
The second part of (ii)
also follows, since an edge $(x,y)$ with $v_{x,y} < w^-_{x,y}$ will never be used in an optimal path.
\end{proof}

Choose $M$ according to Lemma \ref{expdecaylemma}. Now let $N_\beta$ be the number of edges $e$ in the 
interval $[0,n]$ such that $|e|\geq n^\beta$ and $v_e>M|e|$. From the last part of the Lemma, 
\begin{equation}\label{nbeta}
\PP(N_\beta=0 \text{ and } \ell_n\geq n^{\beta}) \leq e^{-c_2n^\beta}.
\end{equation}
But also 
\[
\EE(N_\beta)\leq \sum_{k=n^\beta}^n n\PP\left(v>Mk\right),
\]
since there are at most $n$ edges of any given length $k$ in the interval $[0,n]$. 
Since $\PP(v>Mk)\sim k^{-s}f(k)$ by assumption, we can use
(\ref{asymptotic3}) with $x=n^\beta$ and $r(x)=x^{1/\beta}$ to get
\[
\EE(N_\beta)\leq const\cdot n^{1-\beta(s-1)}f(n^\beta).
\]
Since $f$ is slowly varying, this tends to 0 as $n\to\infty$ whenever $\beta>1/(s-1)$. 
Hence for all such $\beta$, $\PP(N_{\beta}>0)\to0$ as $n\to\infty$; 
combining with (\ref{nbeta}) we have 
\[
\PP\left(\frac{\log \ell_n}{\log n}\geq \beta\right)\to 0 \text{ as } n\to\infty
\text{ for all } \beta>\frac{1}{s-1},
\]
as required for the upper bound.

\noindent \textbf{Lower bound.}
Fix $K>0$ (to be chosen later) and  
let $R_\beta$ be the number of edges within the interval 
$\left(\left\lceil\frac{2n}5\right\rceil, \left\lfloor\frac{3n}5\right\rfloor\right)$
which satisfy $v_e\geq K|e|$ and $|e|\geq n^\beta$. Then
\begin{align}
\nonumber
\EE(R_\beta)
&\geq \sum_{k=n^\beta}^{n/12} \frac{n}{12} \PP\left(\alpha_{0,1}v\geq Kk\right)
\\
\label{rbeta}
&\geq const\cdot n^{1-\beta(s-1)}f(n^\beta).
\end{align}
The first inequality holds since for any $k$ with $n^\beta\leq k\leq n/12$, there are
at least $n/12$ edges of length $k$ within 
$\left(\left\lceil\frac{2n}5\right\rceil, \left\lfloor\frac{3n}5\right\rfloor\right)$,
and the second follows again from (\ref{asymptotic3}).

The RHS of (\ref{rbeta}) tends to infinity as $n\to\infty$ if $\beta<1/(s-1)$. 
Since the corresponding events for different edges $e$ are independent,
we obtain that $\PP(R_\beta\geq 1)\to 1$, i.e.\ with high probability,
at least one such edge exists.

If so, let $e^*=(x^*, y^*)$ be the longest such edge. 
Then define the interval $I^*=(x^*-2|e^*|, y^*+2|e^*|)$,
which is centred on $e$ but is five times as long. 
Note that $I^*$ is still contained in $[0,n]$. 
Finally, let $w^*$ be the maximal weight of a path
contained in the interval $I^*$ which uses only
edges shorter than $e^*$.

We claim that, if $R_\beta\geq 1$, then at least one of the following events must 
hold:
\begin{itemize}
\item[(a)]Some edge at least as long as $e^*$ (maybe $e^*$ itself) is used in the optimal path from 0 to $n$.
\item[(b)]$w^*\geq K|e^*|$.
\item[(c)]There is either no strongly connected point in $(x^*-|e^*|, x^*)$ or there is no strongly connected point in $(y^*, y^*+|e^*|)$. 
\end{itemize}
For if (b) does not hold, then using the edge $e^*$ 
is preferable to any combination of edges 
in $I^*$ which are shorter than $e^*$.
If in addition $(c)$ fails, then using 
appropriate strongly connected points one can 
include the edge $e^*$ simultaneously 
with any 
edge set of compatible edges 
which are wholly to the left of $x^*-|e^*|$ 
or wholly to the right of $y^*+|e^*|$. 
Then the only reason not to use $e^*$ 
is if the optimal path contains an edge $(r,s)$
where either $r<x^*-2|e^*|$, $s>x^*-|e^*|$ 
or $r<y^* + |e^*|$, $s>y^*+2|e^*|$. 
But such an edge has length at least $e^*$. 
So indeed (a) then holds. 

We already have $\PP(R_\beta\geq 1)\to 1$ as $n\to\infty$.
Since the strongly connected points form a renewal process 
with positive density and the renewal intervals have 
exponential tails (see (\ref{strongexp})), 
and $|e^*|\geq n^\beta$, event (c) has probability 
tending to 0 as $n\to\infty$.

If we can show that the probability of event (b) 
also goes to 0, then with probability tending to 1, 
event (a) occurs. Then indeed $\ell_n \geq n^{\beta}$,
and we will have shown that for any $\beta<1/(s-1)$,
$\PP(\ell_n\geq n^\beta)\to 1$ as $n\to \infty$ as required.

So, we need to prove the following:
\[
\textit{Claim: for appropriate $K$, }
\PP(R_\beta\geq 1, w^*\geq K|e^*|)\to 0 \textit{ as } n\to\infty.
\]
Suppose $R_\beta\geq 1$, and condition on the identity of the edge $e^*$.
Let $m=|e^*|$. From the definition of $e^*$, knowing the identity of 
$e^*$ has given us no information about the weights of edges shorter than $m$. 
Then since $I^*$ has length $5m$, the distribution of $w^*$ is
dominated by the distribution of $w_{0,5m}$ in the case $p=1$. 
But the SLLN in that case gives $\frac{1}{5m}w_{0,5m}\to C^{(p=1)}$ 
in probability, for some constant $C^{(p=1)}$. Hence the claim
holds for any $K>C^{(p=1)}$. 

This completes the argument for the longest edge $\ell_n$,
and we can use those results to give the corresponding statements
for the heaviest weight $h_n$. 

The lower bound follows immediately from the bound for $\ell_n$ and property
(\ref{longedgebound}). 
For the upper bound, suppose $\beta>1/(s-1)$. Take $\beta'\in(1/(s-1), \beta)$. 
We know that as $n\to\infty$, the probability that the optimal 
path uses an edge as long as $n^{\beta'}$ tends to 0. 
But also the probability that there exists an edge of length 
less than $n^{\beta'}$ with weight as high as $n^\beta$
is bounded above by
\[
\sum_{k=1}^{n^{\beta'}}n\PP\left(v\geq n^\beta\right)
\leq const\cdot n^{1+\beta'}n^{-\beta s}f(n^\beta)
\]
which converges to 0 as $n\to\infty$. So indeed the probability 
that an edge as heavy as $n^{\beta}$ is used goes to 0, as required.




Now we turn to the fluctuations of $w_{0,n}$ when $2<s<3$. Suppose $\beta<1/(s-1)$,
and choose $\beta'\in(\beta, 1/(s-1))$.
Let $\bare$ be the heaviest edge used in the optimal path from $0$ to $n$. 
We know from above that with high probability $v_\bare> n^{\beta'}$.

Condition on the identity of $\bare$ and the weight of all the other edges in
$[0,n]$, but not the weight of $\bare$ itself. Write $\cA$ for the collection of all this information.

Given $\cA$, we have a lower bound, $V_{\min}$ say, for the weight of $\bare$
(since given the weights of all other edges, $\bare$ will be the heaviest weight in 
the optimal path if and only if its weight exceeds some threshold). 
Now the conditional distribution of $v_{\bare}$ given $\cA$ is the
distribution of a typical weight $v$ conditioned on $v>V_{\min}$. 
Given $\cA$, the value of $w_{0.n}-v_{\bare}$ is constant. 

Certainly we will have $V_{\min}>n^{\beta'}/2$ with high probability. 
For if $V_{\min}\leq n^{\beta'}/2$, then 
$\PP(v_\bare\leq n^{\beta'}|\cA)\geq \PP(v\leq n^{\beta'}|v\geq n^{\beta'}/2)$,
which does not go to 0 as $n\to\infty$ (since the tail of the distribution of $v$ is regularly varying,
so that $P(v\leq 2x|v\geq x)$ converges to a non-zero limit as $x\to\infty$).

But again since $v$ has a regularly varying tail, we have that 
$\PP(v \in [x_n, x_n+n^\beta]|v \geq v_{\min})$ goes to 0 as $n\to\infty$
uniformly in $x_n$ and in $v_{\min}>n^{\beta'}/2$, since $n^{\beta}=o(n^{\beta'})$. 

Hence for some function $\epsilon(n)$ tending to 0 as $n\to\infty$,
we have that with high probability as $n\to\infty$,
\[
\PP(v_\bare\in [x_n, x_n+n^\beta] | \cA) < \epsilon(n) \text{ for all } x_n,
\]
and hence also with high probability
\[
\PP(w_{0,n}\in [y_n, y_n+n^\beta] | \cA) < \epsilon(n) \text{ for all } y_n.
\]

Now we can average over $\cA$, to give that for any sequence $y_n$,
the unconditional distribution of $w_{0,n}$ satisfies
\begin{equation*}
\PP(w_{0,n}\in [y_n, y_n+n^\beta]) \to 0 \text{ as } n\to\infty.
\end{equation*}
as required. 
In particular, with $s<3$ this implies that $\var(w_{0,n})$ grows faster than $n$, 
and that no central limit theorem such as that in Theorem \ref{CLT} 
can hold (even for single values of $t$). 
\end{proof}

Now we want to present two examples that show that for $s=3$ both critical cases are possible: it might happen that the longest edge is $o(\sqrt{n})$ and it is possible that the longest edge satisfies $\frac{\ell_n}{\sqrt{n}} \rightarrow \infty$ in probability as $n \rightarrow \infty$. Let us first look at the case where $f(x) = \frac{1}{\log x}$. Then we have $\mathbb{E} \left[ v^3 \right] = \infty$.
\begin{example}
With $\mathbb{P} \left[ v > k \right] = \frac{1}{k^3 \log k}$ we have $\mathbb{E} \left[ v^3 \right] = \infty$ since $\int_{2}^{\infty} \frac{1}{x \log x} dx = \infty$, but on the other hand (again using (\ref{asymptotic3}))
\begin{align*}
\mathbb{E} \left[ N_{\frac{1}{2}} \right]
&\leq n \sum_{k=\sqrt{n}}^{n} \frac{1}{(Mk)^3 \log Mk} \\
&\leq const \cdot \frac{1}{\log \left( \sqrt{n} \right)} \\
&\xrightarrow[n \rightarrow \infty]{} 0
\end{align*}
So in this case we have that although $\mathbb{E} \left[ v^3 \right] = \infty$ we will not see edges of length $\sqrt{n}$.
\end{example}

However, if $f$ is increasing and such that $\mathbb{E} \left[ N_{\frac{1}{2}} \right] \xrightarrow[n \rightarrow \infty]{} \infty$ then we have $\frac{\ell_n}{\sqrt{n}} \rightarrow \infty$ in probability by the same arguments as in the proof of (\ref{upperlower}) in Lemma \ref{longestedge}. An example is the case where and $f(x) = \log x$:

\begin{example}
Let $\mathbb{P} \left[ v > k \right] = \frac{\log k}{k^3}$. Then the expected number of edges of at least length $\sqrt{n} \log \log n$ and weight at least $M$ times their length is bounded from below by
\begin{align*}
\sum_{k = \sqrt{n} \log \log n}^{\frac{n}{2}} \frac{n}{2} \mathbb{P} \left[ v > Mk \right]
&\geq const \cdot n \left( \sqrt{n} \log \log n \right)^{-2} \cdot \log \left( \sqrt{n} \log \log n \right) \\
&= const \cdot \frac{\frac{1}{2} \log{n} + \log \log \log n}{ \left( \log \log n \right)^2} \\
&\xrightarrow[n \rightarrow \infty]{} \infty
\end{align*}
By the same arguments as for the lower bound in the proof of (\ref{upperlower}) in Lemma \ref{longestedge} we have that with positive probability we will use an edge of length $\sqrt{n} \log \log n$, so $\frac{\ell_n}{\sqrt{n}} \rightarrow \infty$ in probability.
\end{example}

\subsection{Non-constant edge probabilities}

In this section we want to discuss briefly the situation in which the
probabilities that edges are present are not given by a constant $p
\in \left. \left( 0,1 \right. \right]$, but by a sequence $\left( p_i
\right)_{i \geq 1}$ where $p_i$ is the probability that an edge of
length $i$ is present. In the case with constant edge weights this
situation was analyzed in \cite{denisovfosskonstantopoulos}, and it can be
extended to our case as follows. As in \cite{denisovfosskonstantopoulos}
we need the following two conditions:
\begin{align}
&[C1] \; 0 < p_1 < 1 \notag \\
&[C2] \; \sum_{k=1}^{\infty} \left( 1 - p_1 \right) \ldots \left( 1 - p_k \right) < \infty. \notag
\end{align}
Under these conditions the set of strongly connected points is almost surely infinite and the set of strongly connected points forms a stationary renewal process. Since this is all we needed to establish that the set of renewal points is almost surely infinite, conditions $[C1]$ and $C[2]$ are sufficient to get that
$$ \cR \text{ is almost surely an infinite set}.$$
However, in the proof of the strong law of large numbers and the central limit theorem above, we used that the strongly connected points $\tau_i$ have exponential moments. This is in general no longer the case if we replace the constant $p$ by a sequence $\left( p_i \right)_{i \in \mathbb{N}}$. In \cite{denisovfosskonstantopoulos} it was proven that $\mathbb{E} \left[ \tau_0 \right] < \infty$ if the condition
\begin{align}
&[C3] \; \sum_{k=1}^{\infty} k \left( 1 - p_1 \right) \ldots \left( 1 - p_k \right) < \infty \notag
\end{align}
holds. In the proofs of Propositions \ref{Enu} and \ref{Emu} we used that certain errors, see (\ref{m-x}) and (\ref{taum-x}), decay exponentially because the $\tau_i$ had exponential moments. If we are however only interested in showing that $\Gamma_0$ has a finite first moment, then it is enough if these errors decay fast enough to give us finite first moments of $\nu$ and $\mu$ conditioned on $\mu < \infty$. For these errors to decay fast enough it is sufficient to have two moments of $\tau_1 - \tau_0$ and for this it is enough to have $\mathbb{E} \left[ \tau_0 \right] < \infty$. So under $[C1]$, $[C2]$, $[C3]$ and the condition $\mathbb{E} \left[ v^3 \right] < \infty$ for the weights, we still get a SLLN and a CLT for the weight $w_{0,n}$. This agrees with the results in \cite{denisovfosskonstantopoulos}: if the weights are constant then conditions $[C1]$, $[C2]$ and $[C3]$ give us a SLLN and CLT.

\section{Proofs for the model with $\mathbb{E} \left[ v^2 \right] = \infty$}

In this section the weights have a distribution which does not have a second moment, i.e. $\mathbb{E} \left[v^{2}\right] = \infty$. We want to prove Theorems \ref{main1} and \ref{main2}. Again, we start with the case $p=1$ (Theorem \ref{main1}) and then look at the case $p<1$ (Theorem \ref{main2}). The proofs follow closely those in \cite{hamblymartin} where analogous results for directed last-passage percolation in two dimensions were established.

\subsection{Proof of Theorem \ref{main1}}
\begin{proof}[Proof of Theorem \ref{main1}:]
  To prove Theorem \ref{main1} we use approximations of $w_{0,n}$ and
  $w$ that use only the $k$ largest weights. We define
  \begin{equation} \notag \mathcal{C}^k = \left\{ A \subset \left\{
        1,2,\ldots,k \right\} : Y_i \sim Y_j \text{ for all } i,j \in
      A \right\}
\end{equation}
and
\begin{equation} \notag \mathcal{C}_{0,n}^k = \left\{ A \subset
    \left\{ 1,2,\ldots,k \wedge \binom{n+1}{2} \right\} : Y_i^{(n)} \sim Y_j^{(n)}
    \text{ for all } i,j \in A \right\}
\end{equation}
and put
\begin{equation} \notag
w^k = \sup_{A \in \mathcal{C}} \sum_{i \in A, i \leq k} M_i,
\end{equation}
\begin{equation} \notag
w_{0,n}^k = \sup_{A \in \mathcal{C}_{0,n}} \sum_{i \in A, i \leq k} M_i^{(n)}.
\end{equation}
We also define appropriately rescaled versions 
\begin{align*}
\wtw^k_{0,n}&=\frac{w^k_{0,n}}{b_n},
\\
\wtw_{0,n}&=\frac{w_{0,n}}{b_n}.
\end{align*}
The tails of $w$ and
$w_{0,n}$ are bounded by
\begin{equation} \notag S^k = \sup_{A \in \mathcal{C}} \sum_{i \in A,
    i > k} M_i \; \text{ and } \; S^k_{0,n} = \sup_{A \in
    \mathcal{C}_{0,n}} \sum_{i \in A, i > k} M_i^{(n)}.
\end{equation}
The following Lemma implies that $w$ is almost surely finite and that
$w^k \rightarrow w$ for $k \rightarrow \infty$.
\begin{lemma}
  With probability 1 we have $S^k < \infty$ for all $k \geq 0$ and
  $S^k \rightarrow 0$ for $k \rightarrow \infty$.
\end{lemma}
\begin{proof}
  Define $\Lambda_i = \sup_{A \in \mathcal{C}} | A \cap
  \left\{1,\ldots,i \right\} |$. This is the largest number of the edges
  $Y_1,\ldots,Y_i$ that can be included simultaneously in an admissible path. This is
  independent of the weights $\left( M_i \right)_{i \in
    \mathbb{N}}$. In the two-dimensional case in \cite{hamblymartin}
  the corresponding random variable $L_i$ had the distribution of the
  length of the longest increasing subsequence of a random permutation
  of the set $\left\{1,\ldots,i\right\}$. In our case the distribution
  is slightly different, but we get the same asymptotic behaviour and
  the same bounds that we need to prove the Lemma. In the
  two-dimensional case two points $(i,j),(i',j')$ are compatible if
\begin{itemize}
 \item $i \leq i'$ and $j \leq j'$ or $i' \leq i$ and $j' \leq j$
\end{itemize}
In our case two edges (represented by two points in $[0,1]^2$) are compatible if
\begin{itemize}
\item $i \leq \min(i',j')$ and $j \leq \min(i',j')$ or $i' \leq
  \min(i,j)$ and $j' \leq \min(i,j)$ (under the condition $i < j$ and
  $i' < j'$ this is equivalent to $j \leq i'$ or $j' \leq i$)
\end{itemize}
We can see that the second condition is more restrictive. Therefore,
any path that is admissible in our model is also admissible in the
two-dimensional model. If we therefore look at the largest number of
points $Y_1,\ldots,Y_i$ that we can include in an admissible path in
the two models we get that
\begin{equation}\label{Llambda} 
L_i \geq \Lambda_i \text{ a.s.} 
\end{equation}
Since we have $\mathbb{E} \left[ L_i \right] \leq c \sqrt{i}$,
$\mathbb{E} \left[ L_i^2 \right] \leq ci$ and
$\frac{L_i}{i^{\frac{1}{s}}} \xrightarrow[]{d} 0$, for $i \rightarrow
\infty$ and $s \in (0,2)$, we get the same results for the
corresponding variable $\Lambda_i$:
$$ \mathbb{E} \left[ \Lambda_i \right] \leq c \sqrt{i}, \; 
\mathbb{E} \left[ \Lambda_i^2 \right] \leq ci, \; \frac{\Lambda_i}{i^{\frac{1}{s}}} 
\xrightarrow[]{d} 0 \text{ for } i \rightarrow \infty \text{ and } s \in (0,2). $$
Indeed, $\Lambda_i$ has been studied in its own right as the
``independence number of a random interval graph'';
see for example \cite{JusSchWin} and \cite{BouFer},
where, among other things, a central limit theorem and large deviations principle 
are obtained.

The difference between $L_i$ and $\Lambda_i$ is the main difference
between our model and the two-dimensional nearest-neighbour last-passage 
percolation model. 
Since we have the bound (\ref{Llambda}), we can follow the proof 
of Lemma
3.1 in \cite{hamblymartin}: put $U_k = \sum_{i=k+1}^{\infty} \Lambda_i
\left( M_i - M_{i+1} \right)$ and for fixed $A \in \mathcal{C}$ define
$R_i = | A \cap \left\{1,\ldots,i \right\} |$. Then
\begin{align*}
  \sum_{i \in A, i > k} M_i
  &= \lim_{n \rightarrow \infty} \sum_{i \in A, k < i \leq n} M_i \displaybreak[0] \\
  &= \lim_{n \rightarrow \infty} 
\sum_{i = k+1}^{n} M_i \mathbbm{1}_{ \left\{ i \in A \right\} } \displaybreak[0] \\
  &= \lim_{n \rightarrow \infty} \sum_{i = k+1}^{n} M_i \left( R_i - R_{i-1} \right) \displaybreak[0] \\
  &= \lim_{n \rightarrow \infty} 
\left[ - M_{k+1}R_k + \sum_{i = k+1}^{n-1} R_i \left( M_i - M_{i+1} \right) + M_n R_n \right] \displaybreak[0] \\
  &\leq \lim_{n \rightarrow \infty} 
\sum_{i = k+1}^{n-1} R_i \left( M_i - M_{i+1} \right) 
+ \liminf_{n \rightarrow \infty} M_n R_n \displaybreak[0] \\
  &\leq \lim_{n \rightarrow \infty} \sum_{i = k+1}^{n-1} \Lambda_i 
\left( M_i - M_{i+1} \right) + \liminf_{n \rightarrow \infty} M_n \Lambda_n \displaybreak[0] \\
  &= U_k + \liminf_{n \rightarrow \infty} M_n \Lambda_n \displaybreak[0] \\
  &= U_k
\end{align*}
Therefore, we have $S^k \leq U_k$ for all $k$ and it suffices to show
that $U_k \rightarrow 0$ as $k \rightarrow \infty$. Since $U_k$ is the
remainder of an infinite sum, it is actually enough to show that $U_k
< \infty$ almost surely. By the independence of $\left( \Lambda_i
\right)$ and $\left( M_i \right)$ we get that
\begin{align*}
  \mathbb{E} \left[ U_k \right]
  &= \sum_{i=k+1}^{\infty} \mathbb{E} 
\left[ \Lambda_i \right] \left( \mathbb{E} \left[ M_i \right] - \mathbb{E} 
\left[ M_{i+1} \right] \right) \\
  &\leq \sum_{i=k+1}^{\infty} c \sqrt{i} \left( \mathbb{E} \left[ M_i
    \right] - \mathbb{E} \left[ M_{i+1} \right] \right)
\end{align*}
Now we can use the known distribution of the $M_i$ ($M_i$ has the
distribution of $\left(V_i\right)^{-\frac{1}{s}}$ where $V_i \sim
\textrm{Gamma}(i,1)$) to get
$$ \mathbb{E} \left[ U_k \right] \leq \frac{c}{s} 
\sum_{i=k+1}^{\infty} \sqrt{i} \left( i - \frac{1}{s} - 1 \right)^{-\frac{1}{s}} $$
(see Lemma 3.1 of \cite{hamblymartin} for the details). The last sum is
finite for all $k > \frac{1}{s}$ and it follows that $U_k < \infty$
almost surely for all $k$.
\end{proof}
Theorem \ref{main1} then follows from the next two Propositions which
were proved in \cite{hamblymartin} (Propositions 3.2 and 3.3). The
proofs are almost identical and rely on the two following facts:
\begin{itemize}
\item The distribution of the weights $M_i$ in \cite{hamblymartin} and
  in our paper is exactly the same
\item The definition of compatible points/edges is slightly different,
  but such that $\Lambda_i \leq L_i$
\end{itemize}
\begin{proposition} \label{prop1} Let $\varepsilon > 0$ and $k$ be
  fixed. Then for all sufficiently large $n$ there exists a coupling
  of the continuous and the discrete model indexed by $n$ such that
  \begin{equation} \notag \mathbb{P} \left[ \sum_{i=1}^k \left| M_i -
        \widetilde{M}_i^{(n)} \right| > \varepsilon \right] \leq
    \varepsilon,
\end{equation}
\begin{equation} \label{euclid}
\mathbb{P} 
\left[ \sum_{i=1}^k \left\| Y_i - Y_i^{(n)} \right\| > \varepsilon \right] \leq \varepsilon,
\end{equation}
\begin{equation} \notag
\mathbb{P} \left[ \mathcal{C}^k_{0,n} \neq \mathcal{C}^k \right] \leq \varepsilon.
\end{equation}
\end{proposition}
Here we use the Euclidean distance in $\mathbb{Z}^2$ as distance
between two edges $Y_1 = (a,b)$ and $Y_2 = (a',b')$ in (\ref{euclid}).
\begin{proof}[Sketch of the Proof:]
  The first two statements follow straightforwardly from the convergence stated in
  (\ref{Yeq}) and (\ref{Meq}). The
  last statement follows from the fact that with high probability, a small perturbation of
  the $Y_i$ does not affect the ordering of the points.
\end{proof}
\begin{proposition} \label{prop2} Let $\varepsilon > 0$. Then for
  sufficiently large $k$ and $\widetilde{S}^k_{0,n} =
  \frac{S^k_{0,n}}{b_n}$,
\begin{equation} \notag
\mathbb{P} \left[ \widetilde{S}^k_{0,n} > \varepsilon \right] \leq \varepsilon
\end{equation}
for all $n$.
\end{proposition}
\begin{remark}
  A detailed proof of Proposition \ref{prop2} can be found in
  Section 3.2 of \cite{hamblymartin}. The transfer of the proof to our situation
  follows again from the two facts stated before Proposition
  \ref{prop1}.
\end{remark}
We can then write
\begin{align*}
  \left| w - \wtw_{0,n} \right|
  &= \left| \left( w - w^{k_n} \right) + \left( w^{k_n} - \widetilde{w}^{k_n}_{0,n} \right) 
+ \left( \widetilde{w}^{k_n}_{0,n} - \widetilde{w}_{0,n} \right) \right| \\
  &\leq S^{k_n} + \left| w^{k_n} - \widetilde{w}^{k_n}_{0,n} \right| +
  \widetilde{S}^{k_n}_{0,n}
\end{align*}
and for some suitable sequence $k_n$ we have that the first and last
term tend to $0$ in probability. We also have that on
$\mathcal{C}^{k_n}_{0,n} = \mathcal{C}^{k_n}$,
\begin{equation} \notag \left| w^{k_n} - \widetilde{w}^{k_n}_{0,n}
  \right| \leq \sum_{i=1}^{k_n} \left| M_i - \widetilde{M}_i^{(n)}
  \right|
\end{equation}
holds. Since $\mathbb{P} \left[ \mathcal{C}^k_{0,n} \neq \mathcal{C}^k
\right] \rightarrow 0$ and $\sum_{i=1}^{k_n} \left| M_i -
  \widetilde{M}_i^{(n)} \right| \rightarrow 0$ in probability,
we have that $\wtw_{0,n}\to w$ as required for Theorem 
\ref{main1}.
\end{proof}

\subsection{Proof of Theorem \ref{main2}}
The proof in the case $p<1$ goes through in an essentially identical way, 
after making a couple of appropriate observations.

First, the number of edges in the interval $[0,n]$ 
is no longer $\binom{n+1}{2}$, but is now 
a Binomial$\left(\binom{n+1}{2}, p\right)$ random variable. 
Since under (\ref{F}) we have that
\[
a_{\binom{n+1}{2}}/a_{p\binom{n+1}{2}}\to p^{-1/s} \text{ as } n\to\infty,
\]
it's easy to obtain that equation (\ref{Meq}) generalises for $p\in(0,1]$ to
\begin{equation} \label{Meqp}
p^{-1/s}\left( \widetilde{M}_1^{(n)}, \widetilde{M}_2^{(n)}, \ldots , \widetilde{M}_k^{(n)} \right) 
\xrightarrow[]{d} \left( M_1 , M_2 , \ldots , M_k \right) 
\end{equation}
as $n\to\infty$, so that the asymptotics of the heaviest edges change
simply by a constant factor. 

The second issue concerns the set of feasible paths. 
Since not all edges are present, it is no longer the case that
if $x=(i,j)$ and $y=(i',j')$ are two edges with
$i<j\leq i'<j'$, then $x$ and $y$ can necessarily be used in 
the same path (there may be no feasible path between $j$ and $i'$).

However, if $k$ is fixed and $n\to\infty$, 
then with high probability, any subset of the $k$ heaviest edges
which are compatible with each other in this sense can be used together in a path
(since the minimal distance between the endpoints of two such edges
goes to infinity in probability). Then, since the argument above 
shows that we can obtain an arbitrarily close approximation to $w$ by considering
only the $k$ heaviest edges, the result of Theorem \ref{main2} for $p<1$ 
can be obtained just as before. 

\section*{Acknowledgements}
JM and SF were supported by the EPSRC. PS was supported by a ``DAAD
Doktorandenstipendium'' and the EPSRC. We would like to thank the
Department of Statistics, University of Oxford, the Isaac Newton
Institute for Mathematical Sciences, University of Cambridge, and the
Mathematische Forschungsinstitut Oberwolfach for their support.
We thank the referee for a careful reading of the manuscript and 
for valuable comments and suggestions. 
\bibliographystyle{plain}
\bibliography{LPP31}

\bigskip

\noindent
\begin{minipage}[t]{7cm}
\small \sc
Serguei Foss \\
School of Mathematical Sciences\\
Heriot-Watt University\\
Edinburgh EH14 4AS, UK\\
E-mail: {\tt foss@ma.hw.ac.uk}\\
Institute of Mathematics\\
Novosibirsk, Russia
\end{minipage}
\begin{minipage}[t]{6cm}
\small \sc
James Martin \\
Department of Statistics\\
University of Oxford\\
1 South Parks Road\\
Oxford OX1 3TG\\
E-mail: {\tt martin@stats.ox.ac.uk}
\end{minipage}
\\[5mm]
\begin{minipage}[t]{6cm}
\small \sc
Philipp Schmidt \\
Department of Statistics\\
University of Oxford\\
1 South Parks Road\\
Oxford OX1 3TG\\
E-mail: {\tt schmidt@stats.ox.ac.uk}
\end{minipage}

\end{document}